\documentclass[12pt,reqno]{amsart}
\usepackage{amsmath}
\usepackage{txfonts}
\usepackage{amsfonts}
\usepackage{a4wide}
\usepackage{amsmath,amsthm,amssymb,amscd}
\usepackage{latexsym}
\usepackage{hyperref}
\usepackage[numbers,sort&compress]{natbib}
\usepackage{hypernat}
\usepackage{color,enumitem,graphicx}
\usepackage{soul}
\usepackage{CJKutf8}
\usepackage{ulem}
\allowdisplaybreaks
\usepackage{cancel}
\usepackage{comment}
\numberwithin{equation}{section}

\newtheorem{theorem}{Theorem}[section]
\newtheorem{proposition}{Proposition}[section]
\newtheorem{corollary}{Corollary}[section]
\newtheorem{lemma}{Lemma}[section]

\theoremstyle{definition}

\newtheorem{remark}{Remark}[section]

\begin{document}
\title[Asymptotic behavior of ground state solutions]
{Asymptotic behavior of ground state solutions to  nonlinear elliptic problems with the fractional Laplacian}

\author{Jinge, Yang \textsuperscript{1}$^{, \ast}$ ,\,\, Jianfu, Yang \textsuperscript{2} }

\thanks{\textsuperscript{1} School of Science,
Jiangxi University of Water Resources and Electric Power, Nanchang, Jiangxi 330099, People's Republic of China}
\thanks{\textsuperscript{2}  Department of Mathematics, Jiangxi Normal University, Nanchang, Jiangxi 330022, People's Republic of China}

\thanks{Email:  jgyang2007@yeah.net(J.G. Yang)}
\thanks{Email: jfyang200749@sina.com(J.F. Yang)}
\thanks{*Corresponding author}
\begin{abstract}
 In this paper, we consider   the asymptotic behavior of the ground state solution $u_s$ of the nonlinear fractional Laplacian equation
 \begin{equation}\label{eq:0.1a}
  (-\Delta)^su+Vu=|u|^{p-2}u\quad x\in \mathbb{R}^n
  \end{equation}
  by taking  $s$ as a parameter, where $n\geq 4$,  $2<p<\frac{2n}{n-2}$, $V$ is a potential function. We show that for a fixed $p$,   there exists  $s_0\in(0,1)$ such that  equation \eqref{eq:0.1a} admits a ground state solution $u_s$ if   and only if  $s_0<s<1$. Our main results give a description of  the asymptotic behavior of $u_s$ as $s\uparrow1$ and $s\downarrow s_0$: $u_s$ converges to a function as  $s\uparrow1$, and it blows up as $s\downarrow s_0$. Particularly, we prove that $u_s$ concentrates at a minimum point of the function $V$ as  $s\downarrow s_0$. The local uniqueness of $u_s$ is also given.



\end{abstract}

\maketitle

\section{Introduction}

\bigskip

In this paper, we study  the asymptotic behavior of ground state solutions of the following nonlinear fractional Laplacian equation
\begin{equation}\label{eq:1.1}\tag{$P_s$}
(-\Delta)^su+Vu=|u|^{p-2}u,\quad {\rm in} \quad \mathbb{R}^n
\end{equation}
 by taking $s$ as a parameter, where $0<s<1$, $2<p<2^*:=\frac{2n}{n-2}$ and $n\geq 4$. The fractional Laplacian $(-\Delta)^s$ is defined  by
\begin{equation}\label{eq:1.2}
(-\Delta)^su(x)=D(n,s) \text{ P.V.}\int_{\mathbb{R}^n}\frac{u(x)-u(y)}{|x-y|^{n+2s}}\,dy,
\end{equation}
where  P.V. stands for the Cauchy principal value and
\begin{equation}\label{eq:1.3}
D(n,s)=\pi^{-\frac {n}2}2^{2s}\frac{\Gamma(\frac{n+2s}{2})}{\Gamma(2-s)}s(1-s).
\end{equation}
Equivalently, the fractional Laplacian $(-\Delta)^s$ can also be defined on the Schwartz space via the Fourier transform
\[
\mathcal{F}[(-\Delta)^su](\xi)=(2\pi|\xi|)^{2s} \mathcal{F}(u)(\xi),
\]
where $\mathcal{F}(u)(\xi)$ is the Fourier transform of $u$,  that is,
\[
\mathcal{F}(u)(\xi)=\int_{\mathbb{R}^n}e^{-2\pi ix\cdot \xi}u(x)\,dx.
\]
For more details on the definitions of the fractional Laplacian, we refer to \cite{BCP,CS}.

There has been a long history in the studies of semiclassical states of the nonlinear Schr\"{o}dinger equations
\begin{equation}\label{eq:1.3a}
\begin{cases}
-\varepsilon^2\Delta u +V(x)u=|u|^{p-2}u,\quad {\rm in} \quad \mathbb{R}^n,\\
u(x)\to 0\quad {\rm as }\quad |x|\to \infty
\end{cases}
\end{equation}
with various types of distinct concentration behaviors as $\varepsilon \to 0$. The essential relation of the concentrating solutions of \eqref{eq:1.3a} and the
 critical points of the potential function $V$ has been revealed in \cite{FW, O}.  Extensive researches for the problem have been carried on along this way.  Single peaked and multi-peaked solutions were given, for instance, in the works of \cite{ABC, AMS, BPW, BO,  PF, PF1, G} and references therein.

In a bounded domain $\Omega\subset \mathbb{R}^n$, for $n\geq 3$ and $\varepsilon>0$, the asymptotic behavior of solutions $u_\varepsilon$ of the near critical problem
\begin{equation*}
\begin{cases}
-\Delta u =u^{2^*-1-\varepsilon}\quad {\rm in} \quad \Omega,\\
u>0\quad \quad \quad \quad \, \,\, {\rm in} \quad\Omega,\\
u(x)= 0\quad \quad \quad \,\,{\rm on} \quad \partial\Omega\\
\end{cases}
\end{equation*}
was first studied   in \cite{AP} in the unit ball in $\mathbb{R}^3$ as $\varepsilon\to 0$,  and  in \cite{BP}  for the spherical
 domains.  The results for non-spherical domains was given  in \cite{H}.
  A counterpart research on the problem in $\mathbb{R}^n$ can be found in \cite{PW, W93,W}.

Analogous investigations on the fractional Schr\"{o}dinger equations
\begin{equation}\label{eq:1.3b}
\begin{cases}
\varepsilon^{2s}(-\Delta)^s u +V(x)u=|u|^{p-2}u,\quad {\rm in} \quad \mathbb{R}^n;\\
u(x)\to 0\quad {\rm as }\quad |x|\to \infty
\end{cases}
\end{equation}
as well as
\begin{equation}\label{eq:1.3c}
(-\Delta)^s u +V(x)u=|u|^{2^*_s-2-\varepsilon}u\quad {\rm in} \quad \mathbb{R}^n
\end{equation}
were given in \cite{AH, Am,CW, FQT,FS} etc. Here $2^*_s=\frac{2n}{n-2s}$ is the fractional Sobolev exponent.

In contrast to previous works, we will use in this paper the index $s$ as parameter in \eqref{eq:1.1} instead of $\varepsilon$ appeared in \eqref{eq:1.3b} and \eqref{eq:1.3c}, in the  study of the concentration of ground states of problem \eqref{eq:1.1}. By a ground state solution to  \eqref{eq:1.1}, we mean a solution of \eqref{eq:1.1} with the least  energy.
Such a problem is of particular interest in fractional quantum mechanics \cite{Laskin2000, Laskin2002},  Bose-Einstein condensation \cite{Ertik2012} and  dual Airy beams \cite{Longhi2015} and so on. Numerically, it was shown in \cite{Ertik2012} that the parameter $s$ varies according to the different dilute gas in Bose-Einstein condensation.

In mathematics, in the case $V=1$,  it was proved  in \cite{FLS} the  uniqueness of ground state solutions to \eqref{eq:1.1} by studying asymptotic behaviour of the ground state solutions as $s\to1$; we also refer to \cite{FV} for the local uniqueness of ground state solutions to \eqref{eq:1.1} for  $s$ close to $1$.

\bigskip

   Under various assumption on  potential $V$,  the existence of ground state solutions to \eqref{eq:1.1}  can be proved by the mountain pass theorem. Alternatively, the problem can also be solved by searching minimizers of the minimization problem
\begin{equation}\label{eq:1.4}
S_V(n,s)=\inf_{u\in H^s(\mathbb{R}^n)\setminus\{0\}}E_{V,s}(u),
\end{equation}
where
\[
E_{V,s}(u)=\frac{\int_{\mathbb{R}^n}|(-\Delta)^{s/2}u|^2\,dx+\int_{\mathbb{R}^n}Vu^2\,dx}
{\Big(\int_{\mathbb{R}^n}|u|^p\,dx\Big)^{2/p}}.
\]

More precisely, if there exists a nonnegative minimizer $w_s$ of $S_V(n,s)$, then $w_s$ satisfies
\[
(-\Delta)^sw_s+Vw_s=\lambda_s|w_s|^{p-2}w_s\quad {\rm in} \quad \mathbb{R}^n
\]
for some $\lambda_s>0$.  So $u_s=\lambda_s^{\frac{1}{p-2}}w_s$ solves \eqref{eq:1.1}. Moreover, $E_{V,s}(u_s)=E_{V,s}(w_s)=S_V(n,s)$. Therefore, $u_s$ is also a minimizer of $S_V(n,s)$. This implies that $u_s$ is a ground state solution of \eqref{eq:1.1}. Hence, there exists a minimizer of $S_V(n,s)$ which is also a ground state solution of \eqref{eq:1.1} if $S_V(n,s)$ is achieved. In general, if $u_s$ is a ground state solution of \eqref{eq:1.1}, then $u_s$ is a minimizer of $S_{V,s}$ and solves \eqref{eq:1.1}; and vice versa, see Lemma \ref{groundminimizer}.
It  was proved that $S_V(n,s)$ is achieved if $V$ is a positive constant in  \cite{DPV, FQT, FLS},
and in \cite{AFM, C, CW, Si} if $V$ is a non-constant potential.

Moreover, by Proposition 5.1.1 in \cite{DMV},   Proposition 2.1.9 in \cite{S} and Lemma 4.4 in \cite{CS}, one knows that if $V$ is bounded, then the ground state solution $u_s$ of \eqref{eq:1.1} is positive, bounded and belongs to  $C^{1, \,\beta}(\mathbb{R}^n)$ as well as $\lim_{x\to\infty}u_s(x)=0$.

In this paper, we fix $p\in(2,2^*)$. To insure  $p\in(2, 2^*_s)$   so that variational method can apply, we take $s$ as a parameter and find
 \[
 s_0=\frac{(p-2)n}{2p}
 \]
such that  $s_0<s<1$ if and only if $p\in(2, 2^*_s)$.
\bigskip

 We assume throughout this paper that  the potential function $V$ satisfies\\

$(V1)$ $V\in C^{1}(\mathbb{R}^n)$, \quad $0<V_0\leq V(x)< V_\infty:=\lim_{x\to\infty}V(x)<+\infty$;

\vspace{5pt}
$(V2)$ The function $x\cdot \nabla V(x)$  is bounded in  $\mathbb{R}^n$.

\bigskip

It is known from \cite{CW} that $S_V(n,s)$  is achieved by a function $u_s$ if $(V1)$ is fulfilled. We choose in the sequel that $u_s$ is also a ground state solution of \eqref{eq:1.1}. Thus,
\begin{equation}\label{keyequality}
\int_{\mathbb{R}^n}|(-\Delta)^{s/2}u_s|^2\,dx+\int_{\mathbb{R}^n}Vu_s^2\,dx=\int_{\mathbb{R}^n}|u_s|^p\,dx=S_V(n,s)^{\frac{p}{p-2}}.
\end{equation}

The purpose of this paper is to figure out  the asymptotic behaviour of $u_s$ as $s\uparrow 1$ and $s\downarrow s_0$.

Let
$$\mathfrak{C}= \{u\in H^1(\mathbb{R}^n):u \text{ is a nonnegative ground state solution of}\,\,  (P_1).\} $$

Apparently,  $\mathfrak{C}\not=\emptyset$. But it may contain many points due to the lack of uniqueness
for ground state solutions of  $(P_{1})$.

For any $\{s_k\}\in (s_0,1)$, correspondingly,  we have a sequence of ground state solutions $\{u_{s_k}\}$ of  $(P_{s_k})$.
If $s_k\uparrow 1$, then up to a subsequence,  $u_{s_k}$ converges in the following sense.

\begin{theorem}\label{thm:1.3a} Let $n\geq4$ and $2<p<2^*$.  Assume that $V$ satisfies $(V1)$ . Then  there exists  $u\in \mathfrak{C}$  \ such that
\begin{equation*}
\lim_{s_k\uparrow1}u_{s_k}=u
\end{equation*}
in  $L^p(\mathbb{R}^n)$ and $ C_{loc}(\mathbb{R}^n)$.
\end{theorem}

\bigskip

Since the ground state solution $u_s$ belongs to different Sobolev spaces for different $s$, and there is a lack of the compactness for the functional $E_{V,s}(u)$, the study of the asymptotic behavior of $u_s$ as $s\to 1$ and $s\to s_0$ becomes complicated. To avoid these difficulties, taking the advantage that $u_s$ is the ground state solution of  \eqref{eq:1.1}, we may use the pointwise convergence. As to this problem,  among other things, there are two significant ingredients involved. One is the regularity theory for non-local operators, another one is to bound constants appeared in the a priori estimates
uniformly in $s$.

\bigskip

In order to study the limiting behavior of $u_s$ as $s\downarrow s_0$, we remark that problem \eqref{eq:1.4} is related to
the critical minimization problem
\begin{equation}\label{eq:1.7}
S(n,s)=\inf_{u\in {H}^s(\mathbb{R}^n)\setminus\{0\}}\frac{\|(-\Delta)^{s/2}u\|_{L^2(\mathbb{R}^n)}^2}{\|u\|_{L^
{2_{s}^*}(\mathbb{R}^n)}^2}
\end{equation}
 as a limiting problem, which is  achieved by the function
\begin{equation}\label{eq:1.9a}
U(x)=c(\mu^2+(x-x_0)^2)^{-\frac{n-2s}{2}}, \quad x\in \mathbb{R}^n,
\end{equation}
where $c\in \mathbb{R}$, $\mu>0$ and $x_0\in \mathbb{R}^n$. The best fractional Sobolev constant $S(n,s)$  can be worked out as follows:
\begin{equation}\label{eq:1.8}
S(n,s)=2^{2s}\pi^s\frac{\Gamma(\frac{n+2s}{2})}{\Gamma(\frac{n-2s}{2})},
\end{equation}
see  \cite{CLO} for details.  By \cite{CT}, the critical equation
\begin{equation*}
(-\Delta)^{s}u=u^{2_{s}^*-1} \quad \text{ in } \quad\mathbb{R}^n,\quad
u(0)=1
\end{equation*}
has a unique positive solution
\begin{equation}\label{eq:1.9}
Q_{s}=\Bigg(1+\frac{|x|^2}{\lambda_s^2}\Bigg)^{\frac{2s-n}2},\quad {\rm where}\quad \lambda_s=2\Bigg(\frac{\Gamma\big(\frac{n+2s}{2}\big)}{\Gamma\big(\frac{n-2s}{2}\big)}\Bigg)^{\frac12}.
\end{equation}
We remark that
 \begin{equation*}
  p=2_{s_0}^*=\frac{2n}{n-2s_0}.
 \end{equation*}
In particular, the critical equation
\begin{equation*}
(-\Delta)^{s_0}u=u^{2_{s_0}^*-1} \quad \text{ in } \quad\mathbb{R}^n,\quad
u(0)=1
\end{equation*}
has a unique positive solution $Q_{s_0}$ satisfying
\begin{equation}\label{Qs0}
\|(-\Delta)^{s_0/2}Q_{s_0}\|_{L^2(\mathbb{R}^n)}^2=\|Q_{s_0}\|_{L^
{2_{s_0}^*}(\mathbb{R}^n)}^{2_{s_0}^*}=S(n,s_0)^{\frac{p}{p-2}}.
\end{equation}

Now we study the asymptotic behavior of $u_s$ as $s\downarrow s_0$.

We remark that if $0<s\leq s_0$ and  $\inf_{x\in \mathbb{R}^n}V>0,\, V\in L^\infty(\mathbb{R}^n)$, then $S_V(n,s)$ is not achieved, see Lemma \ref{lem:A1} in the Appendix.   Hence, $s_0$ is the thresholds  of the existence and non-existence of minimizers for $S_V(n,s)$. So it would be interested in studying the behavior of $u_s$ as   $s\downarrow s_0$.

Let $x_s\in \mathbb{R}^n$ be  a maximum point of  $u_s$.  Define
\begin{equation}\label{scaling1}
v_s(x):=\mu_s^{{\alpha_s}}u_s(\mu_s x+x_s)
\end{equation}
with
\begin{equation}\label{scaling2}
\mu_s^{-{\alpha_s}}:=\|u_s\|_\infty=u_s(x_s),\quad {\rm  and}\quad {\alpha_s}:=\frac{2s}{p-2}.
\end{equation}

Observe that there does not exist ground state solutions of \eqref{eq:1.1} when $s=s_0$, so one expects that ground state solutions $\{u_s\}$ will blow up, that is $\|u_s\|_{L^\infty(\mathbb{R}^n)}\to \infty$  as $s\downarrow s_0$. This fact will be verified in the proof of Theorem \ref{thm:1.1}. Based on this fact, we will analyze the asymptotic behavior of $\mu_s^{{\alpha_s}}u_s(\mu_s x+x_s)$ as $s\downarrow s_0$.

\bigskip

\begin{theorem}\label{thm:1.1} Let $n\geq4$, $2<p<2^*$ and assume  $(V1)$ and $(V2)$. Then for $s_0<s<1$, we have

$(i)$\begin{equation*}
\lim_{s\downarrow s_0}\mu_s^{{\alpha_s}}u_s(\mu_s x+x_s)= Q_{s_0}
\end{equation*}
in  $L^p(\mathbb{R}^n)$  and uniformly in  $\mathbb{R}^n;$

\bigskip

$(ii)$ $\lim_{s\downarrow s_0}V(x_{s}) = \inf_{x\in \mathbb{R}^n}V(x)$;

\bigskip

$(iii)$ there exists $A_{n,s_0}>0$ depending only on $n$ and $s_0$ such that
\begin{equation*}
\lim_{s\downarrow s_0}(s-s_0)\mu_s^{-2s}=A_{n,s_0}\inf_{x\in \mathbb{R}^n}V(x);
\end{equation*}

$(iv)$ For $s>s_0$ and $s$ close to $s_0$, there exists $C>0$ independent of $s$ such that
\begin{equation*}
\mu_s^{{\alpha_s}}u_s(\mu_s x+x_s)\leq \frac{C}{(1+|x|)^{n-2s}}, \quad x\in \mathbb{R}^n.
\end{equation*}
\end{theorem}

\bigskip

By Theorem \ref{thm:1.1},  for $s\downarrow s_0$, $u_s$ can be written as
\begin{equation*}
 u_s(x)=\Big[A_{n,s_0}\inf_{x\in \mathbb{R}^n}V(x)\Big]^{\frac{1}{p-2}}(s-s_0)^{-\frac{1}{p-2}}Q_{s_0}\Bigg(\Big[A_{n,s_0}\inf_{x\in \mathbb{R}^n}V(x)\Big]^{-\frac{1}{2s}}(s-s_0)^{-1/(2s)}(x-x_s)\Bigg)+o(1).
\end{equation*}
Moreover, the limit point of $x_s$ stays in the set $\mathcal{V}_m$ of minimum points of $V(x)$.

\bigskip

In the proof of Theorem \ref{thm:1.1},  after proving that
$$
\lim_{s\downarrow s_0}S_V(n,s)=S(n,s_0)\quad{\rm and}\quad \lim_{s\downarrow s_0}\|u_s\|_p^p= S(n,s_0)^{\frac p{p-2}},
$$
we show that $\mu_s^{{\alpha_s}}u_s(\mu_s x+x_s)\to Q_{s_0}$ in $L^p(\mathbb{R}^n)$. The convergence of $\mu_s^{{\alpha_s}}u_s(\mu_s x+x_s)\to Q_{s_0}$ in $C_{loc}(\mathbb{R}^n)$ follows from Harnack inequalities and the  Schauder  estimate with the uniform universal constants in $s$  for the equation $(-\Delta)^su=f\in L^\infty(\mathbb{R}^n)$ whenever $s>s_0>0$.

In order to locate the limit point of $x_s$ as $s\downarrow s_0$, we establish uniformly decaying laws of $u_s(x+x_s)$ and $\mu_s^{{\alpha_s}}u_s(\mu_s x+x_s)$ in order.

First, we derive from the Harnack inequality for
$$(-\Delta)^su(x)= a(x)u+b(x)
$$
and
$$
\mu_s^{{\alpha_s}}u_s(\mu_s x+x_s)\to Q_{s_0}\quad {\rm in}\quad  L^p(\mathbb{R}^n)
$$
as $s\downarrow s_0$ that
$u_s(x+x_s)\to 0$ uniformly as $s\downarrow s_0$, and then
$$|u_s(x+x_s)|\leq C|x|^{-n}\quad {\rm for}\quad|x|\geq R.
$$

Next, we establish a $L^\infty$ uniform estimate in $s$ for  nonnegative solution $u$ of
$$(-\Delta)^su(x)\leq a(x)u(x),$$
and  then, an application of the Kalvin transform
yields that
$$\mu_s^{{\alpha_s}}u_s(\mu_s x+x_s)\leq C(1+|x|)^{-(n-2s)}.$$

Inspired of \cite{WZ}, the  conclusion $$V(x_s)\to\inf_{x\in\mathbb{R}^n}V(x)$$
follows  by the asymptotic behavior of $u_s$ and the estimation for the minimizing problem:
\[
E_{V,s}(u_s)=\inf_{u\in H^s(\mathbb{R}^n)\setminus\{0\}}E_{V,s}(u)\leq E(u_s(\cdot+x_s-x_0)).
\]

We remark  that   the assumption $(V2)$ is used to derive a precise blowup rate of  ground states via a Pohozaev identity, which is not needed in the proof of (i) and (ii) in  Theorem \ref{thm:1.1}.

\bigskip

 Next,  we can show that as $s\downarrow s_0$, up to subsequence, $u_s$ concentrates at a point of $\mathcal{V}_m$ . Precisely, let $\{s_k\}$,  $s_0<s_k<1$, be a sequence such that   $s_k\downarrow s_0$. Correspondingly, we have a sequence $x_{s_k}\subset \mathbb{R}^n$ satisfying $x_{s_k}\to x_0$ as $s_k\downarrow s_0$. By Theorem \ref{thm:1.1}, we know that $x_0$ is actually a minimum point of $V(x)$. That is:

\begin{corollary}\label{cor:1.1a} Assume  $(V1)$--$(V2)$. Let $n\geq 4$   and $2<p<2^*$.  Then we have
 $$
 V(x_0)=\inf_{x\in \mathbb{R}^n}V(x)
 $$
  and
\[
|u_{s_k}|^p(x)\to S(n,s_0)^{\frac p{p-2}}\delta(x-x_0)
\]
in the sense of distribution as $k\to \infty$.
\end{corollary}

\bigskip

Corollary \ref{cor:1.1a} implies that if $V(x)$  has a unique minimum point $x_0$, any ground state solution will concentrate at  $x_0$ as $s\downarrow s_0$.

Finally, we deal with the local uniqueness of ground state solutions.  That is, for $s>s_0$ and close to $s_0$, there exists a unique ground state solution of \eqref{eq:1.1}.

If $V$ is a positive constant, the uniqueness of ground state solutions to \eqref{eq:1.1} can be found in \cite{FLS}. Very recently, the uniqueness result was generalized to  the case $V=|x|^2+\omega$ in \cite{GT} by the approach in \cite{FLS}.  This approach    depends on the explicit form of $V$ and the uniqueness of ground state solutions of $(P_1)$.   It seems not applicable to our case.

 We  use Pohozaev identities to  prove the local uniqueness of ground state solutions to Eq.\eqref{eq:1.1}. This approach is  a  powerful in proving local uniqueness of single or multi-bump solutions to various elliptic equations \cite{CLL,DLY,GLW,GNNT,HT} etc.

 Assume  additionally for $n\geq5$ that

$(V3)$ $V(x)$  has a unique minimum point $x_0$ and
\[
\begin{cases}
V(x)=V(x_0)+A|x-x_0|^m+O(|x-x_0|^{m+1}),\quad x\in B_\tau(x_0);\\
\nabla V(x)=mA|x-x_0|^{m-2}(x-x_0)+O(|x-x_0|^{m}),\quad x\in B_\tau(x_0),
\end{cases}
\]
where $\tau$ is a small positive constant, $m\in(1, n-4s_0)$  and $A>0$.

We note that the critical point $x_0$ of $V(x)$ is degenerate if $m>2$, and  is nondegenerate if $m=2$.  The potential function $V(x)$ is not $C^2(\mathbb{R}^n)$ if $m\in(1,2)$.

Under the assumptions  $(V1)$--$(V3)$ and $n\geq5$, we may obtain the refined asymptotic behavior of $u_s$ and $x_s$, that is
  $$
  \lim_{s\downarrow s_0}(x_s-x_0)\mu_s^{-1}=0
  $$ and
\begin{equation*}
\lim_{s\downarrow s_0}\eta_s^{\alpha_s}u_s(\eta_s x+x_0)= Q_{s_0} \text{ uniformly in } x\in \mathbb{R}^n,
\end{equation*}
where
 \[
\eta_s=[A_{n,s_0}V(x_0)]^{-1/(2s)}(s-s_0)^{1/(2s)},
\]
see Lemma \ref{refined location}.

The uniqueness is proved by showing a contradiction. Indeed, were it not the case,  there would exist $s_k\downarrow s_0$ as $k\to\infty$ such that $u_{s_k,1}\neq u_{s_k,2}$ for $k$ large enough. We will try to prove that
\[
\|w_k\|_{L^\infty(\mathbb{R}^n)}\to 0,
\]
where
\begin{equation*}
w_k(x)=\frac{\eta_{s_k}^{\alpha_{s_k}}u_{s_k,1}(\eta_{s_k}x+x_0)-\eta_{s_k}^{\alpha_{s_k}}u_{s_k,2}(\eta_{s_k}x+x_0)}
{\|\eta_{s_k}^{\alpha_{s_k}}u_{s_k,1}(\eta_{s_k}x+x_0)-\eta_{s_k}^{\alpha_{s_k}}u_{s_k,2}(\eta_{s_k}x+x_0)\|_{L^\infty(\mathbb{R}^n)}}.
\end{equation*}
However, $w_k$ may be sign-changing, we turn to work on  $|w_k|$.  By the Kato inequality, $|w_k|$ satisfies the following inequality
\begin{equation*}
(-\Delta)^{s_k}|{w}_k|
\leq \frac{\tilde{v}_{k,1}^{p-1}-\tilde{v}_{k,2}^{p-1}}{\tilde{v}_{k,1}-\tilde{v}_{k,2}}|{w}_k|.
\end{equation*}
A use of the   Harnack inequality for the inequality
\begin{equation*}
(-\Delta)^su(x)\leq a(x)u,
\end{equation*}
 we  derive by the Kelvin transform that
\begin{equation}\label{decay234}
|{w}_k|\leq C(1+|x|^2)^{-\frac{n-2s}{2}},
\end{equation}
where $C>0$ independent of $s$.  By Pohozaev identities and the  blowup argument, we see  that  $w_k\to 0$ in $C_{loc}^{0,\beta}(\mathbb{R}^n)$. This  with \eqref{decay234} yields $\|w_k\|_{L^\infty(\mathbb{R}^n)}\to 0$, a contradiction to $\|w_k\|_{L^\infty(\mathbb{R}^n)}=1$.

\bigskip
Our local uniqueness result is as follows.
\begin{theorem}\label{thm:1.3}Let $n\geq 5$,  $2<p<2^*$ and $s_0<s<1$.  Assume  $(V1)$--$(V3)$. Then there exists a unique ground state  solution to \eqref{eq:1.1} for $s>s_0$ and close to $s_0$.
\end{theorem}

\bigskip

The assumption $(V3)$ in Theorem \ref{thm:1.3} can be reflexed to\\
$(V4)$ $V(x)$  has  $l$ minimum point $x_i$ and satisfies
\[
\begin{cases}
V(x)=V(x_i)+A_i|x-x_i|^{m_i}+O(|x-x_i|^{m_i+1}),\quad x\in B_\tau(x_i);\\
\nabla V(x)=m_iA_i|x-x_i|^{m_i-2}(x-x_i)+O(|x-x_i|^{m}),\quad x\in B_\tau(x_i),
\end{cases}
\]
where $\tau_i$ are  small constants, $m_i\in(1, n-4s_0)$, $A_i>0$, $m_0>m_1\geq m_2\geq \cdots\geq m_l>1$, and $i=1,2,\cdots, l$.

Under the assumption $(V4)$, the concentration  point  is the flattest minimum point.

\bigskip

This paper is organized as  follows. In section 2, we reprove some elementary weight inequalities in order to get constants involved independent of $s$.
It is shown in section 3  a  Harnack inequality and  Schauder estimates for  fractional Laplacian  with the universal constants independent  $s$.    In section 4, we study  the  asymptotic behaviour of ground state solutions as $s\uparrow1$, while in section 5, we prove that any ground state solution blows up in $L^\infty(\mathbb{R}^n)$ as $s\downarrow s_0$,  and study the  asymptotic behavior of the solutions as $s\downarrow s_0$. In Section 6, we determine  the blow up rate and locate the limit point of $x_s$,  then  we prove Theorem \ref{thm:1.1} and  Corollary \ref{cor:1.1a}. In Section 7, we are devoted to prove Theorem \ref{thm:1.3}. In the Appendix, we prove that  if $0<s\leq s_0$ and  $\inf_{x\in \mathbb{R}^n}V>0,\, V\in L^\infty(\mathbb{R}^n)$, then $S_V(n,s)$ is not achieved.

In the sequel, we denote   $\|\cdot\|_{L^p(\mathbb{R}^n)}$ by $\|\cdot\|_p$ for simplicity,  and use $C$ to represent positive constants, which is independent of $s$.

\bigskip

\bigskip

\section{Preliminary results}

\bigskip

In this section, we  reprove some elementary weight inequalities so as to getting uniform estimates in $s$. We remark that our main concern is to show that the universal constants in these inequalities are independent of $s$.

First we recall the Caffarelli--Silvestre extension  of the fractional Laplacian.

Denote by  $X=(x,y)$  a  point of $\mathbb{R}_+^{n+1}$, where $x\in \mathbb{R}^{n}$ and $y>0$.
Let $D$ be an open set in $\mathbb{R}_+^{n+1}$.  For $s\in (0,1)$,
we say $U\in L^2(D,|y|^{1-2s})$, if
\[
\int_{D}|y|^{1-2s}| U(x,y)|^2\,dX<+\infty.
\]

We define the space $H(D,|y|^{1-2s})$ by
\[
H(D,|y|^{1-2s})=\{U(x,y):U\in L^2(D,|y|^{1-2s}) \text{ and } |\nabla U|\in L^2(D,|y|^{1-2s})\},
\]
and  the spaces $X^s(\mathbb{R}_+^{n+1})$  and $\dot{H}^s(\mathbb{R}^{n})$ as the completion of $C_0^\infty(\overline{\mathbb{R}_+^{n+1}})$ and $C_0^\infty(\mathbb{R}^{n})$, under the norms
\[
\|U\|_{X^s(\mathbb{R}_+^{n+1})}^2=\Big(\int_{\mathbb{R}_+^{n+1}}y^{1-2s}|\nabla U|^2\,dX\Big)^{\frac12}.
\]
and
\[
\|u\|_{\dot{H}^s(\mathbb{R}^{n})}
=\Big(\int_{\mathbb{R}^{n}}|2\pi\xi|^{2s}|\mathcal{F}(u)(\xi)|^2\,d\xi\Big)^{\frac12}
=\Big(\int_{\mathbb{R}^{n}}|(-\Delta)^{\frac{s}{2}}u|^2\,dx\Big)^{\frac12},
\]
respectively.

For $u\in \dot{H}^s(\mathbb{R}^n)$,   the function  $U=E_{2s}(u)$  satisfying
\begin{equation*}
\begin{cases}
div(y^{1-2s}\nabla U(x,y)=0, \text{  }(x,y)\in\mathbb{R}_+^{n+1},\\
U(x,0)=u(x),\text{  }x\in\mathbb{R}^{n}
\end{cases}
\end{equation*}
is referred to   the $2s$-harmonic extension of $u$.
By \cite{BCP, CS0}, the fraction Laplacian can be realized by
\[
(-\Delta)^su(x)=-\kappa_s\lim_{y\to 0^+}y^{1-2s}\frac{\partial U}{\partial y}(x,y),\text{  }x\in\mathbb{R}^{n}
\]
with
\begin{equation}\label{kappa}
\kappa_s=\frac{2^{1-2s}\Gamma(1-s)}{\Gamma(s)}.
\end{equation}

Next,  we recall the $A_2$ Muckenhoupt function.

We say a function  $\omega\in A_2$ if  there exists a constant $C_\omega$ such that, for any ball $B\subset \mathbb{R}^{n+1}$,
\begin{equation}\label{a2}
\Bigg(\frac{1}{|B|}\int_B\omega \,dX\Bigg)\Bigg(\frac{1}{|B|}\int_B\omega^{-1}\, dX\Bigg)\leq C_\omega.
\end{equation}

The following Lemma not only show $|y|^{1-2s}\in A_2$ but also that there exists a positive constant  $C_\omega$, independent of $s$, such that \eqref{a2} holds with  $\omega=|y|^{1-2s}$ for any $s\in [s_1,s_2]\subset(0,1)$.
\begin{lemma}\label{wsi1}
Let  $0<s_1\leq s\leq s_2<1$. There exists a positive contant $C_0>0$, depending only on $n,s_1,s_2$  such that,
for any ball $B\subset \mathbb{R}^{n+1}$,
\begin{equation}\label{wsi1-0}
\Bigg(\frac{1}{|B|}\int_B|y|^{1-2s} \,dX\Bigg)\Bigg(\frac{1}{|B|}\int_B|y|^{-(1-2s)}\, dX\Bigg)\leq C_0.
\end{equation}
\end{lemma}

\begin{proof} It can be verified directly.
\end{proof}

\bigskip

Denote by $d\mu=|y|^{1-2s}dX$  a measure in $\mathbb{R}^{n+1}$, which is defined  for any set $E\in \mathbb{R}^{n+1}$ by

\begin{equation*}
\mu(E)=\int_{E}d\mu=\int_{E}|y|^{1-2s}dX.
\end{equation*}
Hence, $\mu$ is a  positive and finite measure.
Given a function $f\in L^1_{loc}(\mathbb{R}^{n+1};\mu)$, we denote
\[f_E=\frac 1{\mu(E)}\int_Ef\,d\mu.
\]

For $R>0$, let $Q_R=B_R\times (0,R)$ and $\partial'Q_R=B_R\times \{0\}$, where $B_R=B_R(y)$ is a ball in $\mathbb{R}^n$  with the radius $R$ and the center at $y\in\mathbb{R}^n$.

We have the following doubling lemma. The proof is standard, we omit it.
\begin{lemma}\label{doubling}Let  $0<s_1\leq s\leq s_2<1$.
For any ball $\mathbf{B}_r\subset \mathbb{R}^{n+1}$, it holds that
\begin{equation*}
\mu(\mathbf{B}_{2r})\leq 2^{2n+2}C_0\,\mu(\mathbf{B}_r),
\end{equation*}
where $C_0$ is the same as that  in  \eqref{wsi1-0}.
\end{lemma}

\bigskip

Now, we introduce  relevant weighted inequalities.

\begin{lemma}\label{weighted Sobolev inequality}(Weighted Sobolev inequality)
Let  $n\geq2$ and $0<s_1\leq s\leq s_2<1$.  For any $U\in X^s(\mathbb{R}_+^{n+1})$, there exists $C=C(n,s)>0$ which is uniformly bounded from above for $s\in[s_1,s_2]$,  such that
\begin{equation}\label{sobolev}
\Big(\int_0^{+\infty}\int_{\mathbb{R}^n}|U|^{2\gamma}y^{1-2s}\,dX\Big)^{\frac1{2\gamma}}\leq C\Big(\int_0^{+\infty}\int_{\mathbb{R}^n}|\nabla U|^{2}y^{1-2s}\,dX\Big)^{\frac1{2}},
\end{equation}
where $\gamma=1+\frac{2}{n-2s}$. Let
$$C_{m}=\sup_{s\in[s_1,s_2]}C(n,s)>0,
$$
 then \eqref{sobolev} holds with $C_m$ instead of $C$.
\end{lemma}

\begin{proof}
If $s\in [s_1,s_2]\cap (0,\frac12)$,  the inequality \eqref{sobolev} was proved in \cite{Bakry,Nguyen}, and the constant
\begin{equation*}
C=C(n,s)=\Big(\frac1{\pi(n+2-2s)(n-2s)}\Big)^{\frac12}\Big[\frac{2\pi^{1-s}\Gamma(n+2-2s)}{\Gamma(1-s)\Gamma(\frac{n+2-2s}2)}\Big]
\end{equation*}
is uniformly bounded in $[s_1,s_2]$.

The case $s=\frac12$ corresponds to the classical Sobolev inequality. Now we deal with the case $s\in [s_1,s_2]\cap (\frac12,1)$.

For any $R>0$ and $U\in C_c^1(\mathbf{B}_R)$,  By Lemma \ref{wsi1}, $|y|^{1-2s}\in A_2$, so we may deduce by Theorem 1.2 in \cite{Fabes1982} as (3.1.6) in \cite{DMV} that
\begin{equation*}
\Big(\int_{\mathbf{B}_R}|y|^{1-2s}|U|^{2\gamma}\,dX\Big)^{\frac1{2\gamma}}\leq C_1\Big(\int_{\mathbf{B}_R}|y|^{1-2s}|\nabla U|^{2}\,dX\Big)^{\frac1{2}},
\end{equation*}
where $\gamma=1+\frac{2}{n-2s}$, and $C_1=C_1(n,C_0)>0$ is a constant independent of $R$ and $s$. This implies that  \eqref{sobolev} is valid for $U\in C_c^1(\mathbf{B}_R)$.
The density argument then completes the proof.
\end{proof}

The following trace inequality can be found in \cite{BCP}.
\begin{lemma}\label{weighted Sobolev inequality2}
Let $0<s<1$ and $n\geq 2$. Then
\[
\|U(\cdot,0)\|_{2^*_s}^2\leq S_1\int_{\mathbb{R}_+^{n+1}}|\nabla U|^2y^{1-2s}\,dX
\]
for all $U\in X^s(\mathbb{R}_+^{n+1})$, where
\[
S_1(s,n)=\frac
{\Gamma(s)\Gamma\Big(\frac{n-2s}{2}\Big)\Gamma(n)^{\frac{2s}{n}}}
{2\pi^{\frac{2s}{n}}\Gamma\Big(\frac{n+2s}{2}\Big)\Gamma(\frac{n}{2})^{\frac{2s}{n}}}.
\]
\end{lemma}

\bigskip

Finally, we derive a weighted trace type interpolation inequality.
\begin{lemma}\label{trace embedding}
Let $f(x,y)\in C_c^1(Q_R\cup\partial'Q_R)$. Then for  $0<s_1\leq s<1$ and $0<R<1$, we have
\[
\int_{B_R}|f(x,0)|^2\,dx
\leq (1+s_1^{-2})\varepsilon\int_{Q_R}|\nabla f(x,y)|^2y^{1-2s}\,dX
+\varepsilon^{-\frac{2-3s_1}{s_1}\cdot2-1}\int_{Q_R}| f(x,y)|^2y^{1-2s}\,dX,
\]
for any $\varepsilon\in(0,1)$.
\end{lemma}
\begin{proof}
For any $f(x,y)\in C_c^1(Q_R\cup\partial'Q_R)$, we have
\begin{equation}\label{trace embedding-1}
\begin{split}
\int_{B_R}|f(x,0)|^2\,dx
&=-\int_0^R\int_{B_R}\partial_y(|f(x,y)|^2)\,dX\\
&\leq\varepsilon\int_{Q_R}|\nabla f(x,y)|^2y^{1-2s}\,dX
+\varepsilon^{-1}\int_{Q_R}|f(x,y)|^2y^{-(1-2s)}\,dX\\
&:=\varepsilon\int_{Q_R}|\nabla f(x,y)|^2y^{1-2s}\,dX+II.
\end{split}
\end{equation}

 If  $0<s\leq \frac12$ and $0<s_1\leq s\leq s_2<1$, then using the fact that

\begin{equation*}
|y|^{-2(1-2s)}\leq \varepsilon^{\frac{2-3s}{2-4s}\cdot2}|y|^{-(2-3s)}
+\varepsilon^{-\frac{2-3s}{s}\cdot2}\\
\leq \varepsilon^2|y|^{-(2-3s)}+\varepsilon^{-\frac{2-3s_1}{s_1}\cdot2},
\end{equation*}
we deduce
\begin{equation*}
\begin{split}
II\leq \varepsilon\int_{Q_R}|f(x,y)|^2y^{-(1-s)}\,dX
+\varepsilon^{-\frac{2-3s_1}{s_1}\cdot2-1}\int_{Q_R}|f(x,y)|^2y^{1-2s}\,dX.
\end{split}
\end{equation*}

 If  $s>\frac12$ and $0<s_1\leq s\leq s_2<1$, since $0<R<1$ we have
\begin{equation*}
\begin{split}
II=\varepsilon^{-1}\int_{Q_R}|f(x,y)|^2y^{1-2s}y^{-2(1-2s)}\,dX
\leq \varepsilon^{-1}\int_{Q_R}|f(x,y)|^2y^{1-2s}\,dX.
\end{split}
\end{equation*}

Hence, the inequality
\begin{equation*}
\begin{split}
|f(x,y)|^2&=\Big(\int_y^R\partial_yf(x,y)\,dy\Big)^2\leq \frac1{2s}R^{2s}\int_0^R|\nabla f(x,y)|^2y^{1-2s}\,dy
\end{split}
\end{equation*}
and $0<R<1$ yield
\begin{equation}\label{trace embedding-3}
\begin{split}
\int_{Q_R}|f(x,y)|^2y^{-(1-s)}\,dX\leq \frac1{2s}\int_0^Ry^{-1+s}dy
\int_{Q_R}|\nabla f(x,y)|^2y^{1-2s}\,dX
\leq s_1^{-2}\int_{Q_R}|\nabla f(x,y)|^2y^{1-2s}\,dX.
\end{split}
\end{equation}
The conclusion follows from \eqref{trace embedding-1}--\eqref{trace embedding-3}.
\end{proof}

\bigskip

Recall that $S_V(n,s)$ is defined in \eqref{eq:1.4}. It is easy to show the following lemma holds.
\begin{lemma} \label{groundminimizer} Let $0<s\leq 1$. The following statements are equivalent:

(i) $u_s$ is a ground state solution of \eqref{eq:1.1};

(ii)  $u_s$ is a minimizer of $S_V(n,s)$ and solves \eqref{eq:1.1};

(iii) $u_s$ is a minimizer of $S_V(n,s)$ with  $\|u_s\|_p^p=S_V(n,s)^{\frac{p}{p-2}}$.
\end{lemma}

\bigskip
\section{Uniform Harnack inequalities and Schauder estimates}

\bigskip

In this section, we will establish a  Harnack inequality with  uniform constants in $s$ for the fractional Laplacian elliptic problem
\begin{equation}\label{eq:0.2}
(-\Delta)^su(x)=  a(x)u+b(x),
\end{equation}
and uniform Schauder estimates for
\begin{equation}\label{eq:0.3}
 (-\Delta)^su(x)=f.
\end{equation}

For a fixed $s$, Harnack inequalities of  \eqref{eq:0.2} and  Schauder estimates of  \eqref{eq:0.3} are known in \cite{CS, JL, S, TX}. However, in order to get uniform bounds in $s$ for $u_s$, we need uniform  estimates with respect to $s$, which seems  not known.

Now we give a Harnack inequality with a  constant $C>0$ independent of  $s$ for \eqref{eq:0.2}.

\begin{lemma}\label{lem:2.4}
Let $0<r<1$.  Let $u\in H^s(\mathbb{R}^n)$ be nonnegative and satisfy that
\begin{equation}\label{lem A.2-0}
(-\Delta)^su(x)= a(x)u+b(x) \quad{\rm  in}\quad B_r.
\end{equation}
Assume $0<s_1\leq s\leq s_2\leq 1$. If $a, b\in L^q(B_{r})$  with $q>\frac{n}{2s_1}\geq\frac{n}{2s}$ and there exists $C_a>0$ independent of $s$ such that
\[
\|a\|_{L^q(B_{r})}\leq C_a,
\]
then
\begin{equation*}
\sup_{B_{{r}/2}}u\leq C\Big(\inf_{B_{r/2}}u+\|b\|_{L^q(B_{r})}\Big),
\end{equation*}
where $C>0$ is independent of $s$.
\end{lemma}

\bigskip

The conclusion that $C>0$ independent of $s$ is new. It relies on meticulous estimates.

To prove Lemma \ref{lem:2.4}, we start with the following two lemmas.

\begin{lemma}\label{harnack}
Let $r\in (0,1)$ and $0<s_1\leq s\leq s_2<1$. Suppose $U\in H(Q_r,|y|^{1-2s})$ is a  nonnegative function  satisfying
\begin{equation}\label{harnack-1}
\int_{Q_r}y^{1-2s}\nabla U\cdot\nabla \phi\,dX
\leq \kappa_s\int_{\partial'Q_r}\big[a(x)U(x,0)+b(x)\big]\phi(x,0)\,dx
\end{equation}
for any nonnegative $\phi\in H(Q_r,|y|^{1-2s})$,  where $\kappa_s$ is defined in \eqref{kappa}.
If $a\in L^q(B_r)$  and  $b\in L^q(B_r)$, $q>\frac{n}{2s_1}\geq\frac{n}{2s}$, and
\begin{equation}\label{harnack-1aaaa}
\|a\|_q\leq C_{a}
\end{equation}
for some $C_{a}>0$ independent of $s$, then for any $p>0$   and  $\theta, \theta'\in (0,1)$ with $\theta<\theta'$,
there exists $C>0$  independent of $s$ such that
\begin{equation*}
\sup_{Q_{\theta r}}U\leq C\Big[(\theta'-\theta)^{-\gamma/p}\|U\|_{L^{p}(Q_{\theta'r},|y|^{1-2s})}+\|b\|_{L^q(B_{r})}\Big],
\end{equation*}
where $\gamma>0$ depends only on $n$, $s_1$ and $s_2$.
\end{lemma}
\begin{proof} The proof relies on the Moser-Nash iteration, we sketch the proof.
For $k,m>0$,  let $\overline{U}=U+k$ and
\begin{equation*}
\overline{U}_m=
\begin{cases}
\overline{U}&\text{   if }U<m;\\
k+m& \text{   if }U\geq m.
\end{cases}
\end{equation*}
If $b\neq0$,  fix $k$ as
\[
k=\|b\|_q.
\]
Otherwise, choose arbitrary $k>0$ and finally let $k\to0$.

Choose the test function $\phi$ in \eqref{harnack-1} with the form
\[
\phi=\eta^2\overline{U}_m^\beta\overline{U}\in H(y^{1-2s}, Q_r)
\]
for some $\beta\geq 0$ and some nonnegative function
$\eta\in C_c^1(Q_r\cap \partial'Q_r)$.

Let $W=\overline{U}_m^{\frac{\beta}2}\overline{U}$.
 We have that
\begin{equation}\label{083001}
\begin{split}
&\int_{Q_r}y^{1-2s}\nabla U\cdot\nabla \phi\,dX\geq \frac{1}{8(1+\beta)}\int_{Q_r}y^{1-2s}|\nabla(\eta W)|^2\,dX
-3\int_{Q_r}y^{1-2s}W^2|\nabla \eta|^2\,dX.
\end{split}
\end{equation}

Since $0<s_1\leq s\leq s_2<1$ and $\overline{U}>U$, we have for $q'=\frac{q}{q-1}$ that,
\begin{equation*}
\begin{split}
&\kappa_s\int_{\partial'Q_r}\big[a(x)U(x,0)+b(x)\big]\phi(x,0)\,dx
\leq C\big[\|a\|_{L^q(\partial'Q_r)}+1\big]\|\eta W\|_{L^{2q'}(\partial'Q_r)}^2
:=I_1.
\end{split}
\end{equation*}
We may choose  $C>1$ independent of $s$ here and the rest of the proof.

Noting $2<2q'=\frac{2q}{q-1}<2^*_s$, by the Young inequality and Lemma \ref{weighted Sobolev inequality2}, we deduce that
\begin{equation*}
\begin{split}
I_1&\leq C\big[\|a\|_{L^q(\partial'Q_r)}+1\big]\|\eta W\|_{L^{2}(\partial'Q_r)}^{2\theta}\|\eta W\|_{L^{2_s^*}(\partial'Q_r)}^{2(1-\theta)}\\
&\leq \Big[\frac1{16(1+\beta)}\Big]^{\frac1{1-\theta}}\int_{Q_r}y^{1-2s}|\nabla(\eta W)|^2\,dX
+ \Big[\frac1{16(1+\beta)}\Big]^{-\frac1{\theta}}C^{\frac1{\theta}}\big[\|a\|_{L^q(\partial'Q_r)}+1\big]^{\frac1{\theta}}\|\eta W\|_{L^{2}(\partial'Q_r)}^{2}.\\
\end{split}
\end{equation*}
where
\[
1<\frac1\theta=\frac{2sq}{2sq-n}\leq \frac{2s_2q}{2s_1q-n}.
\]
Hence,
\begin{equation}\label{083003a}
I_1\leq\frac1{16(1+\beta)}\int_{Q_r}y^{1-2s}|\nabla(\eta W)|^2\,dX
+\Big[C(1+\beta)\big(\|a\|_{L^q(\partial'Q_r)}+1\big)\Big]^{\frac{2s_2q}{2s_1q-n}}\|\eta W\|_{L^{2}(\partial'Q_r)}^{2}.
\end{equation}
It follows from \eqref{harnack-1}, \eqref{harnack-1aaaa}, \eqref{083001}--\eqref{083003a}, Lemma \ref{weighted Sobolev inequality} and Lemma \ref{trace embedding} that
\begin{equation*}
\begin{split}
\Big(\int_{Q_r}y^{1-2s}|\eta W|^{\nu}\,dX\Big)^{\frac2{\nu}}\leq\int_{Q_r}y^{1-2s}|\nabla(\eta W)|^2\,dX\leq C(1+\beta)^{\alpha}
\int_{Q_r}y^{1-2s}(|\eta|^2+|\nabla \eta|^2 )|W|^2\,dX,
\end{split}
\end{equation*}
where $\nu=2+\frac{4}{n-2s}$. For any $0<r'<r\leq 1$, choose $\eta\in C_c^1(Q_r\cap \partial'Q_r)$ such that
\[
\eta\equiv1\text{ in }Q_{r'}, \text{     } |\nabla\eta|\leq \frac2{r-r'}.
\]
Then
\begin{equation*}
\Big(\int_{Q_{r'}}y^{1-2s}|\overline{U}_m|^{\nu(\beta+2)/2}\,dX\Big)^{\frac2{\nu}}
\leq C\frac{(1+\beta)^{\alpha}}{(r-r')^2}
\int_{Q_r}y^{1-2s}|\overline{U}|^{\beta+2}\,dX
\end{equation*}
provided that $\overline{U}\in L^{\beta+2}(Q_r, y^{1-2s})$. Letting $m\to +\infty$, then we obtain

\begin{equation*}
\|\overline{U}\|_{L^{\nu(\beta+2)/2}(Q_{r'}, y^{1-2s})}\leq \Big(C\frac{(1+\beta)^{\alpha}}{(r-r')^2}\Big)^{1/(\beta+2)}
\|\overline{U}\|_{L^{\beta+2}(Q_{r}, y^{1-2s})}.
\end{equation*}
 The rest proof is standard by iteration.
\end{proof}

\begin{lemma}\label{harnack2}
Let $r\in (0,1)$, $0<s_1\leq s\leq s_2<1$. Suppose $U\in H(Q_r,|y|^{1-2s})$ is a nonnegative function satisfying
\begin{equation}\label{harnack2-1}
\int_{Q_r}y^{1-2s}\nabla U\cdot\nabla \phi\,dX
\geq \kappa_s\int_{\partial'Q_r}\big[a(x)U(x,0)+b(x)\big]\phi(x,0)\,dx
\end{equation}
for any nonnegative function $\phi\in H(Q_r,|y|^{1-2s})$.
If $a, b\in L^q(B_r)$,  $q>\frac{n}{2s_1}\geq\frac{n}{2s}$,  and
\[
\|a\|_q\leq C_{a}
\]
for some $C_{a}>0$ independent of $s$. Then for any $\theta, \theta'\in (0,\sqrt{2}/2)$,  $\theta<\theta'$,
there exist $C>0$  independent of $s$, and a positive constant $p>0$  such that
\begin{equation*}
\inf_{Q_{\theta r}}U+\|b\|_{L^q(B_{r})}\geq C(\theta'-\theta)^{\gamma/p}\|U\|_{L^{p}(Q_{\theta'r},|y|^{1-2s})},
\end{equation*}
where $\gamma>0$ depends only on $n$, $s_1$ and $s_2$.
\end{lemma}
\begin{proof}
Let $\overline{U}=U+k$ with $k>0$ and $V=\Big(\overline{U}\Big)^{-1}$.
If $b\neq0$,  fix $k$ as $k=\|b\|_q$.  If $b=0$, choose arbitrary $k>0$ and finally let $k\to0$.
Let  $\mathbf{B}_{R,0}=\mathbf{B}_{R,0}((y,0))$ be the ball in $\mathbb{R}^{n+1}$, and $\mathbf{B}_{R,0}^+$, $\mathbf{B}_{R,0}^-$ be the half balls in $\mathbb{R}_+^{n+1}$ and $\mathbb{R}_-^{n+1}$, respectively.

We first prove that there exists $p>0$ such that
\begin{equation}\label{harnack2-4}
\int_{Q_{\theta' r}}\overline{U}^{p}|y|^{1-2s}\,dX\int_{Q_{\theta' r}}\Big(\overline{U}\Big)^{-p}|y|^{1-2s}\,dX\leq C.
\end{equation}
Since $Q_{\theta' r}\subset \mathbf{B}_{\sqrt{2}\theta'r,0}^+\subset Q_r$ for $\theta' \in (0, \sqrt{2}/2)$, it suffices to prove that
\begin{equation}\label{harnack2-5}
\int_{\mathbf{B}_{\sqrt{2}\theta'r,0}^+}\overline{U}^{p}|y|^{1-2s}\,dX\int_{\mathbf{B}_{\sqrt{2}\theta'r,0}^+}\Big(\overline{U}\Big)^{-p}|y|^{1-2s}\,dX\leq C.
\end{equation}
Define $\widetilde{U}\in H(B_r\times (-r,r), |y|^{1-2s})$ by
\begin{equation*}
\widetilde{U}(x,y)=
\begin{cases}
 & \overline{U}(x,y)   \,\,\,\quad {\rm if} \quad y\geq0,\\
 & \overline{U}(x,-y)   \quad {\rm if} \quad y<0.
\end{cases}
\end{equation*}
The inequality \eqref{harnack2-5} is equivalent to
\begin{equation}\label{harnack2-5a}
\int_{\mathbf{B}_{\sqrt{2}\theta'r,0}}\widetilde{U}^{p}|y|^{1-2s}\,dX\int_{\mathbf{B}_{\sqrt{2}\theta'r,0}}\Big(\widetilde{U}\Big)^{-p}|y|^{1-2s}\,dX\leq C.
\end{equation}
If  there exists a $p$ such that
\begin{equation}\label{harnack2-6}
\int_{\mathbf{B}_{\sqrt{2}\theta'r,0}}e^{p|W|}|y|^{1-2s}\,dX\leq C,
\end{equation}
where
$$
W=\log\widetilde{U}-\beta \quad {\rm and}\quad \beta=\mu\big(\mathbf{B}_{\sqrt{2}\theta'r,0}\big)^{-1}\int_{\mathbf{B}_{\sqrt{2}\theta'r,0}}|y|^{2s}\log\widetilde{U}\,dX,
$$
then \eqref{harnack2-5a} is valid.

\bigskip

Now, we prove that there exists a $p$ such that \eqref{harnack2-6} holds.
Let $X_0\in \mathbf{B}_{\sqrt{2}\theta'r,0}$ and choose $t>0$ such that $ \mathbf{B}_{2t}(X_0)\subset\mathbf{B}_{\sqrt{2}\theta'r,0}$, and
let $\Psi\in C_c^1( \mathbf{B}_{2t}(X_0))$  be a function such that
\[
0\leq \Psi\leq 1, \quad \Psi=1\text{ in } \mathbf{B}_t(X_0), \quad |\nabla \Psi|\leq \frac{C}t,
\]
Take  the  test function $\phi$ in \eqref{harnack2-1} as
$$
\phi=\overline{U}^{-1}\Psi^2,
$$
we have
\begin{equation}\label{harnack2-22}
\int_{Q_r}|y|^{1-2s}{|\nabla W |^2}\Psi^2\,dX\leq C\int_{Q_r}|y|^{1-2s}{|\nabla \Psi |^2}\,dX
\end{equation}
and
\begin{equation*}
\begin{split}
\int_{\mathbf{B}_t^+(X_0)}|y|^{1-2s}{|\nabla W |^2}\,dX
\leq C\int_{\mathbf{B}_{2t}^+(X_0)}|y|^{1-2s}{|\nabla \Psi |^2}\,dX,
\end{split}
\end{equation*}

Replacing $\Psi$ in \eqref{harnack2-22} by
\[
\tilde{\Psi}(x,y)=\Psi(x,-y),
\]
we have
\begin{equation*}
\begin{split}
\int_{\mathbf{B}_t^-(X_0)}|y|^{1-2s}{|\nabla W |^2}\,dX
\leq C\int_{\mathbf{B}_{2t}^-(X_0)}|y|^{1-2s}{|\nabla \Psi |^2}\,dX,
\end{split}
\end{equation*}

Hence,
\begin{equation*}
\begin{split}
\int_{\mathbf{B}_t(X_0)}|y|^{1-2s}{|\nabla W |^2}\,dX
\leq C\int_{\mathbf{B}_{2t}(X_0)}|y|^{1-2s}{|\nabla \Psi |^2}\,dX\leq Ct^{-2}\int_{\mathbf{B}_{2t}(X_0)}|y|^{1-2s}\,dX \leq
 4Ct^{-2}\int_{\mathbf{B}_{t}(X_0)}|y|^{1-2s}\,dX.
\end{split}
\end{equation*}

Let $W_{X,t}=W_{\mathbf{B}_t(X_0)}$.
The Poincare inequality, see Theorem 1.5 in \cite{Fabes1982}, and Lemma \ref{wsi1} yield for any ball $\mathbf{B}_t(X_0)
\subset \mathbf{B}_{\sqrt{2}\theta'r,0}$ that,
\begin{equation*}
\begin{split}
&\frac{1}{\int_{\mathbf{B}_t(X_0)}|y|^{1-2s}\,dX}\int_{\mathbf{B}_t(X_0)}|y|^{1-2s}|W-W_{X,t}|\,dX\\
&\leq \frac{1}{\int_{\mathbf{B}_t(X_0)}|y|^{1-2s}\,dX}\Bigg(\int_{\mathbf{B}_t(X_0)}|y|^{1-2s}\,dX\int_{\mathbf{B}_t(X_0)}|y|^{1-2s}|W-W_{X,t}|^2\,dX\Bigg)^{1/2}\\
&\leq Ct\Bigg[\frac{1}{\int_{\mathbf{B}_t(X_0)}|y|^{1-2s}\,dX}\int_{\mathbf{B}_t(X_0)}|y|^{1-2s}|\nabla W|^2\,dX\Bigg]^{1/2}\leq C.
\end{split}
\end{equation*}
Thus, 	the weighted John-Nirenberg  inequality, see  Corollary 18.8 in \cite{HK}) and  Lemma \ref{doubling},  there exists  $p>0$ such that
\begin{equation*}
\begin{split}
\int_{\mathbf{B}_{\sqrt{2}\theta'r,0}}e^{p|W|}|y|^{1-2s}\,dX
\leq 3\int_{\mathbf{B}_{\sqrt{2}\theta'r,0}}|y|^{1-2s}\,dX\leq \frac1{2-2s}Cr^{n+2-2s}\leq \frac1{2-2s_2}C,
\end{split}
\end{equation*}
that is, \eqref{harnack2-6} holds.

 Let  $\tilde U=\Big(\overline{U}\Big)^{-1}$. Choosing  the  test function
$$
\phi=\overline{U}^{-2}\Phi\in H(Q_r,|y|^{1-2s})
$$
in \eqref{harnack2-1}, where  $\Phi\in H(Q_r,|y|^{1-2s})$ is nonnegative, we have
\begin{equation*}
\begin{split}
-\int_{Q_r}y^{1-2s}\nabla \tilde U\cdot\nabla \Phi\,dX
\geq \kappa_s\int_{\partial'Q_r}\frac{a(x)U(x,0)+b(x)}{\overline{U}}\tilde U\Phi\,dx.
\end{split}
\end{equation*}
By Lemma \ref{harnack}, for any $p>0$,
\begin{equation*}
\sup_{Q_{\theta r}}\tilde U\leq C(\theta'-\theta)^{-\gamma/p}\|\tilde U\|_{L^{p}(Q_{\theta'r},|y|^{1-2s})}.
\end{equation*}
This with  \eqref{harnack2-4} yields the conclusion.
\end{proof}

{\bf Proof of Lemma \ref{lem:2.4}.}
Let  $U$ be the $2s$-harmonic extension of $u$. Then by \eqref{lem A.2-0},
\begin{equation*}
	\int_{Q_r}y^{1-2s}\nabla U\cdot\nabla \phi\,dX
= \kappa_s\int_{\partial'Q_r}\big[a(x)U(x,0)+b(x)\big]\phi(x,0)\,dx.
\end{equation*}

Taking $\theta=\frac12$ and $\theta'=\frac23$ in Lemmas \ref{harnack} and \ref{harnack2},  we find
\[
\sup_{B_{r/2}}u\leq \sup_ {Q_{r/2}}U\leq C\big(\inf_{Q_{r/2}}U+\|b\|_{L^q(B_r)}\big)\leq C\big(\inf_{B_{r/2}}u+\|b\|_{L^q(B_r)}\big).
\]
The proof is completed.\quad
$\Box$

\bigskip

Next, we establish uniform Schauder estimates with respect to $s\in [s_1,1)\subset(0,1)$ for the fractional Laplacian equation
$$(-\Delta)^su=f.$$
The proof is inspired by  Proposition 2.8 in  \cite{S}, in which the Schauder estimates  obtained for fixed $s\in(0,1)$.

\begin{lemma}\label{schauder}
Assume $(-\Delta)^su=w$  and $u, w\in L^\infty(\mathbb{R}^n)$. Then

$(i)$ If $0<s_1<  s\leq \frac12$, then  $u\in C^{0,\, \beta s_1}(\mathbb{R}^n)$  for any $\beta\in(0,2)$, and
\[
\|u\|_{C^{0,\,\beta s_1}(\mathbb{R}^n)}\leq C(\|u\|_{L^\infty(\mathbb{R}^n)}+\|w\|_{L^\infty(\mathbb{R}^n)}),
\]
 where $C>0$ depends only $n$, $\beta$ and $s_1$.

$(ii)$ If  $\frac12\leq s_1< s< 1$, then $u\in C^{1,s_1-1/2}(\mathbb{R}^n)$   and
\[
\|u\|_{C^{1,s_1-1/2}(\mathbb{R}^n)}\leq C(\|u\|_{L^\infty(\mathbb{R}^n)}+\|w\|_{L^\infty(\mathbb{R}^n)}),
\]
where $C>0$ depends only $n$ and $s_1$.
\end{lemma}
To prove Lemma \ref{schauder}, we need the following lemma,which follows from Propositions 2.5 and 2.6 in  \cite{S}.

\begin{lemma}\label{schauder2}
Let $n\geq1$, $0<s\leq s_2<\frac12$ and $u\in C^{i,\alpha}(\mathbb{R}^n)$ with $\alpha\in (2s_2,1]$, $i=0,1$. Then $(-\Delta)^su\in C^{i,\alpha-2s_2}(\mathbb{R}^n)$ and
\[
\|(-\Delta)^su\|_{C^{i,\alpha-2s_2}(\mathbb{R}^n)}\leq C\|u\|_{C^{i,\alpha}(\mathbb{R}^n)},
\]
where $C>0$ depends only on $n$, $\alpha$ and $s_2$.
\end{lemma}

\bigskip

{\bf Proof of Lemma \ref{schauder}.}
$(i)$ It suffices to prove that there exist $C=C(n,s_1)>0$ and $r>0$  such that for $x\in \mathbb{R}^n$,
\begin{equation*}
\|u\|_{C^{0,\,\beta s_1}(B_r(x))}\leq C(\|u\|_{L^\infty(\mathbb{R}^n)}+\|w\|_{L^\infty(\mathbb{R}^n)}).
\end{equation*}
This is equivalent to
\begin{equation*}
\|v\|_{C^{0,\,\beta s_1}(B_r(0))}\leq C(\|v\|_{L^\infty(\mathbb{R}^n)}+\|f\|_{L^\infty(\mathbb{R}^n)}),
\end{equation*}
where $v=u(\cdot-x)$ and $f=w(\cdot-x)$.

Let $\eta\in C_c^\infty(\mathbb{R}^n)$ be a function such that $\eta=1$ in $B_1$, $\eta=0$ in $B_2^c$ and $0\leq \eta\leq 1$.
Denote
\[
\tilde{v}=(-\Delta)^{-s}(\eta f).
\]
By  Proposition 2.8 in \cite{S}, for any $\beta\in(0,2)$,
\[
\|(-\Delta)^{-\frac12}(\eta f)\|_{C^{0, 1-(2-\beta)s_1}(\mathbb{R}^n)}\leq
C(n,\beta,s_1)[\|(-\Delta)^{-\frac12}(\eta f)\|_{L^\infty(\mathbb{R}^n)}+\|\eta f\|_{L^\infty(\mathbb{R}^n)}].
\]
While
\begin{equation*}
\begin{split}
|(-\Delta)^{-\frac12}(\eta f)|=C(n)\Big|\int_{\mathbb{R}^n}\frac{\eta f}{|x-y|^{n-1}}\,dx\Big|\leq C(n)\int_{B_2}\frac{1}{|x-y|^{n-1}}\,dx\|f\|_{L^\infty(\mathbb{R}^n)},
\end{split}
\end{equation*}
it yields
\[
\|(-\Delta)^{-\frac12}(\eta f)\|_{C^{0, 1-(2-\beta)s_1}(\mathbb{R}^n)}\leq
C(n,\beta,s_1)\| f\|_{L^\infty(\mathbb{R}^n)}.
\]
Hence we derive from  Lemma \ref{schauder2}  that,
\begin{equation}\label{shauder-4}
\begin{split}
\|\tilde{v}\|_{C^{0, \beta s_1}(\mathbb{R}^n)}&=\|(-\Delta)^{\frac12-s}(-\Delta)^{-\frac12}(\eta f)\|_{C^{0, \,\beta s_1}(\mathbb{R}^n)}\\
&\leq C(n,\beta,s_1)\|(-\Delta)^{-\frac12}(\eta f)\|_{C^{0, 1-(2-\beta)s_1}(\mathbb{R}^n)}\\
&\leq
C(n,\beta,s_1)\| f\|_{L^\infty(\mathbb{R}^n)}.
\end{split}
\end{equation}

Set $h=v-\widetilde{v}$, then $h$ solves
\begin{equation*}
\begin{cases}
(-\Delta)^s h=0 &\text{   in  } B_1, \\
h=v-\widetilde{v}       &\text{   in  } B_1^c .
\end{cases}
\end{equation*}
Equation $(4.5)$ in \cite{CLM} implies that
\begin{equation*}
h(x)=\int_{\mathbb{R}^n\setminus{B_1}}P_s(y, x)\Big(v(y)-\tilde{v}(y)\Big)\,dy,\quad x\in B_1
\end{equation*}
where
\begin{equation*}
P_s(y, x)=
\begin{cases}
\begin{aligned}
 &\frac{\Gamma(n/2)}{\pi^{\frac n 2+1}}sin (\pi s)\Big(\frac{1-|x|^2}{|y|^2-1}\Big)^s \frac{1}{|x-y|^n}, \quad |y|>1,  \\
 &0,       \quad |y|\leq 1.
 \end{aligned}
\end{cases}
\end{equation*}
Therefore, if $0<s_1<s<1$ and $x\in B_{1/2}$, we have
\begin{equation*}
\begin{split}
|h(x)|&\leq \int_{\mathbb{R}^n\setminus{B_1}}P_s(y, x)\,dy\Big(\|v\|_{L^\infty(\mathbb{R}^n)}+\|\tilde{v}\|_{L^\infty(\mathbb{R}^n)}\Big)\\
&\leq C(n)\Big[ \int_{1\leq |y|\leq 2}(|y|^2-1)^{-s}\,dy+\int_{ |y|\geq 2}{|y|^{-(n+2s_1)}}\,dy\Big]\Big(\|v\|_{L^\infty(\mathbb{R}^n)}+\|\tilde{v}\|_{L^\infty(\mathbb{R}^n)}\Big)\\
&\leq C(n,s_1)\Big(\|v\|_{L^\infty(\mathbb{R}^n)}+\|\tilde{v}\|_{L^\infty(\mathbb{R}^n)}\Big).
\end{split}
\end{equation*}
Similarly,  for  $k\geq 1$ we have
\begin{equation}\label{shauder-10}
\begin{split}
\|h\|_{C^k(B_{1/2}(0))}\leq C(n,s_1,k)\Big(\|v\|_{L^\infty(\mathbb{R}^n)}+\|\tilde{v}\|_{L^\infty(\mathbb{R}^n)}\Big).
\end{split}
\end{equation}
Combing \eqref{shauder-4} with \eqref{shauder-10}, we get the result:
\begin{equation*}
\begin{split}
\|v\|_{C^{0, \beta s_1}(B_{1/2}(0))}
&\leq \|\tilde{v}\|_{C^{0, \beta s_1}(C^1(B_{1/2}(0))}+\|h\|_{C^1(B_{1/2}(0))}\\
&\leq C(n,\beta,s_1)\Big(\|v\|_{L^\infty(\mathbb{R}^n)}+\|f\|_{L^\infty(\mathbb{R}^n)}\Big).
\end{split}
\end{equation*}

$(ii)$ Since
\[
\tilde{v}=(-\Delta)^{-s} (\eta f)=(-\Delta)^{1-s} (-\Delta)^{-1}(\eta f).
\]
The classical Schauder estimate yields that
\[
(-\Delta)^{-1}(\eta f)\in C^{1,\frac32-s_1}(\mathbb{R}^n)
\]
as well as
\begin{equation*}
\begin{split}
\|(-\Delta)^{-1}(\eta f)\|_{ C^{1,\frac32-s_1}(\mathbb{R}^n)}\leq C\|f\|_{L^\infty(\mathbb{R}^n)}.
\end{split}
\end{equation*}
So by Lemma \ref{schauder2} and $1-s\in (0,1/2]$, we infer that
\begin{equation}\label{AAA}
\begin{split}
\|\tilde{v}\|_{C^{1,s_1-\frac12}(\mathbb{R}^n)}
&=\|(-\Delta)^{1-s}(-\Delta)^{-1}(\eta f)\|_{C^{1,s_1-\frac12}(\mathbb{R}^n)}\\
&\leq C(n,s_1)\|(-\Delta)^{-1}(\eta f)\|_{ C^{1,\frac32-s_1}(\mathbb{R}^n)}\\
&\leq C(n,s_1)\|f\|_{L^\infty(\mathbb{R}^n)}.
\end{split}
\end{equation}
 The conclusion then follows from  \eqref{shauder-10},  \eqref{AAA} and $v=\tilde{v}+h$.\quad
$\Box$

\bigskip

 Next, we  establish a $L^\infty$ bound  for solutions of a differential inequality, which  is uniform in $s\in(0,1)$.  Inspired by Proposition 5.1.1 in \cite{DMV}, first  we derive a higher integrability of the solutions by  the Moser-Nash iteration.

\begin{lemma}\label{lem A.1} Let $n\geq3$.
Suppose that $u\in H^s(\mathbb{R}^n)$, $0<s<1$, is nonnegative and satisfies
\begin{equation}\label{lem:0.1-0}
(-\Delta)^su(x)\leq a(x)u \quad {\rm in }\quad \mathbb{R}^n.
\end{equation}
Then there exists a $\delta(n,s)>0$ such that if
\begin{equation}\label{lem:0.1-1a}
\int_{B_{2r}}|a(x)|^{\frac{n}{2s}}\,dx<\delta(n,s),
\end{equation}
we have
\[
\|u\|_{L^{\frac{(2_s^*)^2}{2}}(B_r)}\leq C(n,s,r)\|u\|_{L^{2_s^*}(\mathbb{R}^n)}.
\]
Moreover, there exist $\delta_0, C_0>0$, independent of $s$,  such that $\delta(n,s)>\delta_0$ and $C(n,s,r)<C_0$.
\end{lemma}
\begin{proof}
Let $q>1$ and $L>0$. Set a convex function
\[
\varphi_{q,L}(t) =
\begin{cases}
0, & \text{if } t< 0, \\
t^q,        & \text{if } 0 \leq t < L, \\
qL^{q-1}(t-L)+L^q,  & \text{if } t \geq L.
\end{cases}
\]

Let $\eta\in C^1_c(\mathbb{R}^n)$ be a function such that $\eta=1$ in $B_{r}$, $\eta=0$ in $B_{2r}^c$, $0\leq \eta\leq 1$, and $|\nabla \eta|\leq \frac{2}{r}$.

We may verify that
\begin{equation}\label{lem:0.1-4}
\begin{split}
 \Big(\int_{B_{2r}}|\eta \varphi_{q,L}(u)|^{2_s^*}\,dx\Big)^{\frac2{2_s^*}}&\leq S^{-1}(n,s)\int_{\mathbb{R}^n}|(-\Delta)^{s/2}(\eta \varphi_{q,L}(u))|^{2}\,dx\\
&\leq qS^{-1}(n,s)\int_{\mathbb{R}^n}a(x)\eta^2(x) |\varphi_{q,L}(u)(x)|^2\,dx\\
&\quad+S^{-1}(n,s){D(n,s)}\int_{\mathbb{R}^n}\int_{\mathbb{R}^n}\frac{|\eta(x)-\eta(y)|^2}{|x-y|^{n+2s}}|\varphi_{q,L}(u)(x)|^2\,dydx.
\end{split}
\end{equation}
By the H\"{o}lder inequality,
\begin{equation*}
\int_{B_{2r}}a(x)\eta^2(x) |\varphi_{q,L}(u)(x)|^2\,dx\leq \Big(\int_{B_{2r}}a(x)^{\frac n{2s}}\,dx\Big)^{\frac {2s}n}\Big(\int_{B_{2r}}|\eta \varphi_{q,L}(u)|^{2_s^*}\,dx\Big)^{\frac2{2_s^*}}.
\end{equation*}
Choose $\delta(n,s)$ such that
\begin{equation}\label{lem:0.1-6}
qS^{-1}(n,s)\delta(n,s)^{\frac{2s}{n}}=\frac12,
\end{equation}
and in the rest of the proof, we assume that \eqref{lem:0.1-1a} holds.

Whence,
\begin{equation}\label{lem:0.1-7}
\Big(\int_{B_{2r}}|\eta \varphi_{q,L}(u)|^{2_s^*}\,dx\Big)^{\frac2{2_s^*}}
\leq 2S^{-1}(n,s){D(n,s)}\int_{\mathbb{R}^n}\int_{\mathbb{R}^n}\frac{|\eta(x)-\eta(y)|^2}{|x-y|^{n+2s}}|\varphi_{q,L}(u)(x)|^2\,dydx
\end{equation}
by \eqref{lem:0.1-4}--\eqref{lem:0.1-6}.

Now we estimate the term on the right hand of \eqref{lem:0.1-7}.

Since $|\nabla \eta|\leq 2r^{-1}$,
\begin{equation*}
\begin{split}
\int_{|x-y|\leq 1}\frac{|\eta(x)-\eta(y)|^2}{|x-y|^{n+2s}}\,dy
&\leq \int_{|x-y|\leq 1}\frac{4r^{-2}|x-y|^2}{|x-y|^{n+2s}}\,dy=\frac{C(n)\pi^{n/2}}{(1-s)r^2}.
\end{split}
\end{equation*}
Noting that $0<\eta\leq 1$, we have
\begin{equation}\label{lem:0.1-9}
\begin{split}
\int_{|x-y|\geq 1}\frac{|\eta(x)-\eta(y)|^2}{|x-y|^{n+2s}}\,dy\leq\int_{|x-y|\geq 1}\frac{4}{|x-y|^{n+2s}}\,dy=\frac{C(n)\pi^{n/2}}{s}.
\end{split}
\end{equation}
Hence,  \eqref{lem:0.1-7}--\eqref{lem:0.1-9} imply that
\begin{equation}\label{lem:0.1-10}
\Big(\int_{B_{2r}}|\eta \varphi_{q,L}(u)|^{2_s^*}\,dx\Big)^{\frac2{2_s^*}}
\leq 2S^{-1}(n,s){D(n,s)}{C(n)\pi^{n/2}}\Big[\frac{1}{(1-s)r^2}+\frac{1}{s }\Big]\int_{\mathbb{R}^n}|\varphi_{q,L}(u)(x)|^2\,dx.
\end{equation}

Setting $q=\frac{2_s^*}2$ and  letting $L\to +\infty$ in \eqref{lem:0.1-10}, we obtain
\begin{equation*}
\begin{split}
\left( \int_{B_{r}} |u|^{\frac{(2_s^*)^2}{2}} \, dx \right)^{\frac{2}{2_s^*}}
&\leq 2S^{-1}(n,s)D(n,s) C(n){4\pi^{n/2}} \left[ \frac{1}{(1-s)r^2} + \frac{1}{s} \right] \int_{\mathbb{R}^n} |u(x)|^{2_s^*} \, dy \, dx\\
&:=C(n,s,r)\int_{\mathbb{R}^n} |u(x)|^{2_s^*}  \, dx.
\end{split}
\end{equation*}
By \eqref{eq:1.3}, \eqref{eq:1.8} and  the definitions of $\delta(n,s)$ and $C(n,s,r)$,  there exist $\delta>0$ and $C_0>0$, independent of $s$,  such that $\delta(n,s)>\delta$ and $C(n,s,r)<C_0$.
\end{proof}

Finally, we establish a uniform $L^\infty$ estimate in $s$ for  the solutions.

\begin{lemma}\label{lem:A.3}
 Let $u\in H^s(\mathbb{R}^n), 0<s<1,$  be a nonnegative solution of \eqref{lem:0.1-0}.

If $a\in L^t(\mathbb{R}^n)$ with $t>\frac{n}{2s}$ and
\[
\|a\|_{L^t(\mathbb{R}^n)}\leq C_a,
\]
then
\[
\|u\|_\infty\leq C(n,t,s,C_a)\|u\|_{2_s^*}.
\]
Moreover, for fixed $s_0\in (0,1)$, if $t>\frac{n}{2s_0}$ and $C_a$ is independent of $s$, $C(n,t,s, C_a)$ is uniformly bounded in $(s_0,1)$.
\end{lemma}
\begin{proof}
Letting $r\to+\infty$ in \eqref{lem:0.1-4}, we have
\begin{equation*}
\begin{split}
\Big(\int_{\mathbb{R}^n}| \varphi_{q,L}(u)|^{2_s^*}\,dx\Big)^{\frac2{2_s^*}}
&\leq qS^{-1}(n,s)\Big(\int_{\mathbb{R}^n}a(x)^t \,dx\Big)^{1/t}\Big(\int_{\mathbb{R}^n}\varphi_{q,L}(u)^{\frac{2t}{t-1}}(x)\,dx\Big)^{1-1/t}.
\end{split}
\end{equation*}
If $u\in L^{\frac{2qt}{t-1}}(\mathbb{R}^n)$, let $L\to+\infty$,  then
\begin{equation}\label{lem:A.3-2}
\|u\|_{q2_s^*}^2\leq q^{1/q}C(n,s,C_a)^{1/q}\|u\|_{\frac{2qt}{t-1}}^2,
\end{equation}
where $C(n,s,C_a):=S^{-1}(n,s)C_a$.

We choose in \eqref{lem:A.3-2} that $q=q_0$ such that $\frac{2q_0t}{t-1}=2_s^*$,  and successively   $q=q_j$, where $q_j=\frac{2_s^*(t-1)}{2t}q_{j-1}=\Big(\frac{2_s^*(t-1)}{2t}\Big)^{j+1}$, $j=1,2,\cdots$.
Then
\begin{equation*}
\begin{split}
\|u\|_{q_j2_s^*}^2&\leq q_j^{1/q_j}[1+C(n,s,C_a)]^{1/q_j}\|u\|_{\frac{2q_jt}{t-1}}^2\leq e^{\sum_{k=0}^j{(\log q_k})/{q_k}}[1+C(n,s,C_a)]^{\sum_{k=0}^j 1/q_k}\|u\|_{2_s^*}^2.
\end{split}
\end{equation*}

Since $\frac{2_s^*(t-1)}{2t}>1$,  letting $j\to +\infty$, we obtain
\[
\|u\|_{\infty}\leq C(n,t,s, C_a)\|u\|_{2_s^*},
\]
where
\[
C(n,t,s, C_a)^2:=e^{\sum_{k=0}^\infty{(\log q_k})/{q_k}}[1+C(n,s,C_a)]^{\sum_{k=0}^\infty 1/q_k}.
\]

Finally,  if  $t>\frac{n}{2s_0}$ and $C_a$ is independent of $s$,  and then $C(n,t,s, C_a)$ is uniformly bounded for $s\in(s_0,1)$.

\end{proof}

\bigskip

\bigskip
\section{Asymptotic behaviour of ground state solutions as $s\uparrow1$}

\bigskip

In this section, we study asymptotic behaviour of ground state solutions of \eqref{eq:1.1} as  $s\uparrow1$.

\bigskip

We recall that $S_{V_\infty}(n,s)$ is defined in \eqref{eq:1.4} with $V=V_\infty$.
\begin{lemma}\label{sjx} There  holds that
\[
\lim_{s\uparrow1}S_{V_\infty}(n,s)=S_{V_\infty}(n,1).
\]
\end{lemma}
\begin{proof}
By \cite{BL,DPV}, there exists  a positive radial minimizer $\varphi_s$ of $S_{V_\infty}(n,s)$ and  $\varphi_s$ solves \eqref{eq:1.1} with $V=V_\infty$ for any $s\in (s_0,1]$. Since $\varphi_s\to \varphi_1$
in $L^p(\mathbb{R}^n)$ as $s\uparrow1$(see \cite{{FLS}}) and \eqref{keyequality} with $V=V_\infty$, the conclusion follows.

\end{proof}

\bigskip

Let $u_s$   be a ground state solution $u_s$ of  \eqref{eq:1.1} , which is also minimizer of $S_V(n,s)$, $s_0<s<1$. We have a similar result.

\begin{lemma}\label{s1} There holds
\begin{equation}\label{limit}
\lim_{s\uparrow1}S_V(n,s)=S_V(n,1).
\end{equation}
Moreover, for any $\{s_k\}\uparrow 1$, the corresponding minimizer sequence $\{u_{s_k}\}$ converges to a minimizer $v$ of $S_V(n,s)$,  $\|v\|_p^p=1$, in the sense that
 \[
 \lim_{k\to\infty}\frac{u_{s_k}}{\|u_{s_k}\|_p}=v
 \]
     in   $L^p(\mathbb{R}^n)$  and   $ C_{loc}(\mathbb{R}^n)$ up to subsequence.
\end{lemma}

\begin{proof}

We first prove that
\begin{equation}\label{sjx-1}
\limsup_{s\uparrow1}S_{V}(n,s)\leq S_{V}(n,1).
\end{equation}
Let $\varphi$ be a positive minimizer of $S_{V_\infty}(n,1)$ .  It is well known that  $\varphi$ is exponential decaying at  infinity. The inequality \eqref{sjx-1} follows from $S_{V}(n,s)\leq E_{V,s}(\varphi)$.

Let
\[
v_s=\frac{u_s}{\|u_s\|_p}.
\]
Then $v_s$ is also  a minimizer of $S_V(n,s)$, which satisfies
\begin{equation}\label{ws}
(-\Delta)^sv_s+Vv_s=S_V(n,s)|v_s|^{p-2}v_s,\quad {\rm in} \quad \mathbb{R}^n.
\end{equation}

 By \eqref{sjx-1}, we deduce that for $s<1$ and closed to $1$, $S_V(n,s)$ is uniformly bounded.
Therefore,
\begin{equation}\label{ws-1}
\int_{\mathbb{R}^n}|(-\Delta)^{s/2}v_s|^2+|v_s|^2\,dx=S_V(n,s)\leq C.
\end{equation}

In the inequality
\begin{equation*}
(-\Delta)^sv_s\leq C|v_s|^{p-2}v_s:=a(x)v_s\quad {\rm in} \quad \mathbb{R}^n,
\end{equation*}
the coefficient $a(x)$ satisfying
\[
\Big[\int_{\mathbb{R}^n}a(x)^{\frac{p}{p-2}}\,dx\Big]^{\frac{p-2}{p}}=C\Big[\int_{\mathbb{R}^n}v_s^p\,dx\Big]^{\frac{p-2}{p}}=1
\]
is also independent of $s\in(s_0,1)$.

Let $s_1=(s_0+1)/2$. Noting
\[
\frac{p}{p-2}=\frac n{2s_0}>\frac n{2s_1},
\]
by \eqref{eq:1.7}, \eqref{eq:1.8} and \eqref{ws-1} as well as Lemma \ref{lem:A.3}, we infer for $s\in (s_1,1)$ that,
\begin{equation}\label{ws-3}
\|v_s\|_\infty\leq C\|v_s\|_{2_s^*}\leq CS(n,s)^{-1}\|(-\Delta)^{\frac s2}v_s\|_2\leq C.
\end{equation}

Applying the Fourier transformation to \eqref{ws}, we obtain
\[
|\xi|^{2s}\hat{v}_s=-\widehat{Vv_s}+S_V(n,s)\widehat{v_s^{p-1}}.
\]
It follows from $(V1)$,  \eqref{ws-1}, \eqref{ws-3} and $\|v_s\|_p^p=1$ that
\begin{equation*}
\begin{split}
\int_{\mathbb{R}^n}|\xi|^{4s}|\hat{v}_s|^2\,d\xi
&=\int_{\mathbb{R}^n}\big|-\widehat{Vv_s}+S_V(n,s)\widehat{v_s^{p-1}}\big|^2\,d\xi\\
&\leq  2\Big[\int_{\mathbb{R}^n}\big|Vv_s\big|^2\,d\xi
+S_V(n,s)^2\int_{\mathbb{R}^n}\big|v_s^{p-1}\big|^2\,d\xi\Big]\\
&\leq 2\Big[V_\infty^2\int_{\mathbb{R}^n}|v_s|^2\,d\xi
+S_V(n,s)^2\|v_s\|_\infty^{p-2}\int_{\mathbb{R}^n}v_s^{p}\,d\xi\Big]\leq C,
\end{split}
\end{equation*}
 Hence,
\begin{equation*}
\|v_s\|_{H^{2s}(\mathbb{R}^n)}\leq C
\end{equation*}
and
\begin{equation}\label{ws-6}
\|v_s\|_{H^{1}(\mathbb{R}^n)}\leq 2\|v_s\|_{H^{2s}(\mathbb{R}^n)}\leq C, \text{ for any  } s>1/2.
\end{equation}
Moreover, by \eqref{ws-1} and  $\|v_s\|_p=1$, we have
\begin{equation}\label{ws-11}
\begin{split}
S_V(n,s)&\geq \frac12 \int_{\mathbb{R}^n}(|(-\Delta)^{s_1/2}v_s|^2+|v_s|^2)\,dx\geq \frac12 S_V(n,s_1)>0.
\end{split}
\end{equation}
So we derive from \eqref{ws-3}, \eqref{ws-6}, \eqref{ws-11} and  Lemma \ref{schauder} that,  for any $\{s_k\}$ with  $s_k \uparrow1$ as $k\to\infty$, there exists  a positive constant $c$ such that $\lim_{k\to \infty}S_V(n,s_k)=c$ and
\begin{equation}\label{s1-2}
\begin{split}
&v_{s_k}\rightharpoonup v\text{        in     }C_{loc}(\mathbb{R}^n);\\
&v_{s_k}\rightharpoonup v\text{        in     }H^1(\mathbb{R}^n);\\
&v_{s_k}\rightharpoonup v\text{        in     }L^p(\mathbb{R}^n);\\
&v_{s_k}\to v\text{        in     }L_{loc}^2(\mathbb{R}^n)
\end{split}
\end{equation}
up to subsequence.

Let $l=\|v\|_p^p$. Obviously, $0\leq l\leq 1$.  For simplicity, we denote $v_{s_k}$ by $v_k$.

We claim that $l>0$. Indeed, if $l=0$, then $v_k\to 0$  in   $L_{loc}^2(\mathbb{R}^n)$.
  By $(V1)$, it yields
\begin{equation}\label{p0}
\lim_{k\to\infty}\int_{\mathbb{R}^n}[V(x)-V_\infty]v_k^2\,dx=0.
\end{equation}
So we deduce from  Lemma \ref{sjx}, $\|v_k\|_p^p=1$ and \eqref{p0} that
\begin{equation*}
\begin{split}
\lim_{k\to\infty}S_V(n,s_k)&=\lim_{k\to\infty}\Big[\int_{\mathbb{R}^n}|(-\Delta)^{s_k/2}v_k|^2\,dx
+\lim_{k\to\infty}\int_{\mathbb{R}^n}V_\infty v_k^2\,dx\Big]+\lim_{k\to\infty}\int_{\mathbb{R}^n}(V-V_\infty)v_k^2\,dx\\
&\geq \lim_{k\to\infty}S_{V_\infty}(n,s_k)=S_{V_\infty}(n,1),
\end{split}
\end{equation*}
which contradicts to \eqref{sjx-1} since $S_{V_\infty}(n,1)>S_{V}(n,1)$. Thus, $l>0$.

Let $\varphi_k=v_k-v$. By \eqref{s1-2}, we have
\begin{equation*}
\begin{split}
&\varphi_{s_k}\rightharpoonup 0\text{        in     }H^1(\mathbb{R}^n);\\
&\varphi_{s_k}\rightharpoonup 0\text{        in     }L^p(\mathbb{R}^n);\\
&\varphi_{s_k}\to 0\text{        in     }L_{loc}^2(\mathbb{R}^n).
\end{split}
\end{equation*}
Then $(V1)$ implies that
\begin{equation}\label{s1-4a}
\lim_{k\to\infty}\int_{\mathbb{R}^n}[V(x)-V_\infty]|\varphi_k|^2\,dx=0.
\end{equation}
By the Brezis-Lieb Lemma, we have
\begin{equation*}
\lim_{k\to\infty}\|\varphi_k\|_p^p=\lim_{k\to\infty}\|v_k\|_p^p-\lim_{k\to\infty}\|v\|_p^p=1-l.
\end{equation*}

By $(V1)$, \eqref{ws-1} and  \eqref{s1-4a},
\begin{equation*}
\begin{split}
\lim_{k\to\infty}S_V(n,s_k)
&=\lim_{k\to\infty}\int_{\mathbb{R}^n}(2\pi|\xi|)^{s_k}|\hat{\varphi}_k|^2\,dx+\int_{\mathbb{R}^n}V_\infty\hat{\varphi}_k^2\,dx
+\lim_{k\to\infty}\int_{\mathbb{R}^n}(2\pi|\xi|)^{s_k}|\hat{v}|^2\,dx+\int_{\mathbb{R}^n}V\hat{v}^2\,dx\\
&+\lim_{k\to\infty}\int_{\mathbb{R}^n}\big(V-V_\infty\big)|\hat{\varphi}_k|^2\,dx\\
&\geq \lim_{k\to\infty}S_{V_\infty}(n,s_k)\|\varphi_{s_k}\|_p^2+ \lim_{k\to\infty}S_V(n,s_k)\|v\|_p^2\\
&\geq \lim_{k\to\infty}S_{V}(n,s_k)(\|\varphi_{s_k}\|_p^2+ \|v\|_p^2).
\end{split}
\end{equation*}
Hence, the fact $\lim_{k\to\infty}S_{V}(n,s_k)>0$ yields
\[
(1-l)^{\frac2p}+ l^{\frac2p}\leq 1.
\]
It implies  $l=1$ since $0<l\leq 1$. Thus, $v_k\to v$ in $L^p(\mathbb{R}^n)$ and
\begin{equation}\label{s1-10}
\begin{split}
\lim_{k\to\infty}S_V(n,s_k)&
=\lim_{k\to\infty}\int_{\mathbb{R}^n}(2\pi|\xi|)^{s_k}|\hat{v}_k|^2\,dx+\int_{\mathbb{R}^n}V{v}_k^2\,dx\\
&\geq \int_{\mathbb{R}^n}(2\pi|\xi|)^{2}|\hat{v}|^2\,dx+\int_{\mathbb{R}^n}V{v}^2\,dx\\
&=E_{V,1}(v)\geq S_V(n,1).
\end{split}
\end{equation}
It follows from \eqref{sjx-1} and \eqref{s1-10} that
\[
\lim_{k\to\infty}S_V(n,s_k)=S_V(n,1)=E_{V,1}(v).
\]
This means $v$ is a minimizer of $S_V(n,1)$ with $\|v\|_p^p=1$.
Due to $\{s_k\}$ is arbitrary selected,  \eqref{limit} holds.
\end{proof}

\bigskip

{\bf Proof of Theorem \ref{thm:1.3a}.} The conclusion follows from Lemma \ref{s1}.

\quad $\Box$
\bigskip

\bigskip
\section{Asymptotic behaviour of ground state solutions as $s\downarrow s_0$}

\bigskip

In this section, we will  prove that any ground state solution $u_s$ of \eqref{eq:1.1} blows up in $L^\infty(\mathbb{R}^n)$ as $s\downarrow s_0$,  and study the asymptotic behaviour of $u_s$  as $s\downarrow s_0$.
\bigskip

Let $S_V(n,s)$ and ${S}(n,s)$ be  defined in \eqref{eq:1.4} and \eqref{eq:1.7}, respectively.
\begin{lemma} \label{lem3.1} $S_V(n,s)$ is uniformly bounded for $s\in (0,1)$, and
\begin{equation*}
\lim_{s\downarrow s_0}S_V(n,s)= S(n,s_0),
\end{equation*}
\end{lemma}

\begin{proof}
For any $\varphi\in C_c^\infty(\mathbb{R}^n)$ such that $\int_{\mathbb{R}^n}|\varphi|^p\,dx=1$,  by   $(V1)$ we find that
\begin{equation*}
\begin{split}
S_V(n,s)
\leq\int_{\mathbb{R}^n}|(-\Delta)^{s/2}\varphi|^2\,dx+\int_{\mathbb{R}^n}V|\varphi|^2\,dx\leq[(2\pi)^2+V_\infty]\|\varphi\|_{H^1(\mathbb{R}^n)}.
\end{split}
\end{equation*}

First,
by \eqref{eq:1.8}, we have
\begin{equation*}
\liminf_{s\downarrow s_0}S_V(n,s)\geq \liminf_{s\downarrow s_0}S(n,s)=S(n,s_0).
\end{equation*}

Next, we show that

\begin{equation*}
\limsup_{s\downarrow s_0}S_V(n,s)\leq S(n,s_0).
\end{equation*}

By Theorem 1.1 in \cite{CT}, there exists a function $\psi_s=(1+|x|^2)^{\frac{2s-n}{2}}$ such that $S(n,s)=E_{0,s}(\psi_s)$.

Let $\varepsilon>0$ and
\begin{equation*}
\psi_{\varepsilon,s}=\varepsilon^{\frac{2s-n}{2}}\psi_s\bigg(\frac x\varepsilon\bigg).
\end{equation*}

There holds that,
\begin{equation*}
\begin{split}
\limsup_{s\downarrow s_0}\inf_{u\in H^s(\mathbb{R}^n)\setminus\{0\}}\,E_{V,s}(u)
&\leq \limsup_{s\downarrow s_0} \,E_{V,s}(\psi_{\varepsilon,s})\\
&= \limsup_{s\downarrow s_0}\,S(n,s)+\frac{\|\sqrt{V}\psi_{\varepsilon,s}\|_{2}^2}{\|\psi_{\varepsilon,s}\|_{2_s^*}^2}\\
&\leq S(n,s_0)+\frac{\varepsilon^{2s_0}V_\infty}{\|\psi_{s_0}\|_{2_{s_0}^*}^2}\int_{\mathbb{R}^n}\Big(1+|x|^2\Big)^{2s_0-n}\,dx,
\end{split}
\end{equation*}
which implies
\[
\limsup_{s\downarrow s_0}S_V(n,s)\leq S(n,s_0).
\]
The conclusion follows.
\end{proof}

\bigskip

\begin{remark}\label{r:3.1}

By  \eqref{keyequality} and Lemma \ref{lem3.1}, both $\|(-\Delta)^{s/2}u_s\|_2^2+\|\sqrt{V}u_s\|_2^2$ and $\|u_s\|_p^p$ are uniformly
bounded for $s\in(s_0,1)$, so is in $H^s(\mathbb{R}^n)$ for $s\in(s_0,1)$ by $(V1)$.  The embedding $H^s(\mathbb{R}^n) \hookrightarrow H^{s_0}(\mathbb{R}^n)$ tells us that  $\{u_s\}$ is uniformly bounded in  $H^{s_0}(\mathbb{R}^n)$.  Hence, we infer from \eqref{eq:1.7} and \eqref{eq:1.8} that  $\|u_s\|_{2_s^*}$ is uniformly bounded with respect to  $s\in(s_0,1)$.  Moreover,
\begin{equation}\label{r:3.1-2}
\lim_{s\downarrow s_0}\|(-\Delta)^{s/2}u_s\|_2^2+\|\sqrt{V}u_s\|_2^2=\lim_{s\downarrow s_0}\|u_s\|_p^p= S(n,s_0)^{\frac p{p-2}}.
\end{equation}
\end{remark}
\bigskip

\begin{lemma}\label{lem:3.2}
Let $u_s$ be a positive ground state solution of \eqref{eq:1.1}. Then
\begin{equation}\label{lem:3.2-1}
\lim_{s\downarrow s_0}\|u_s\|_{\infty}=+\infty.
\end{equation}
\end{lemma}
\begin{proof}

Since $u_s$ is a solution of  \eqref{eq:1.1}, by  $(V1)$, we have
\[
\int_{\mathbb{R}^n}u_s^2(u_s^{p-2}-V_0)\,dx\geq\int_{\mathbb{R}^n}|(-\Delta)^{\frac s2}u_s|^2\,dx\geq0,
\]
which implies that
\begin{equation}\label{lem:3.2-2}
\|u_s\|_{L^\infty(\mathbb{R}^n)}\geq V_0^{\frac1{p-2}}>0.
\end{equation}

In order to  prove \eqref{lem:3.2-1}, we suppose, on the contrary, that there exists a sequence $\{s_k\}\downarrow s_0$  as $k\to\infty$
such that $\{u_{s_k}\}$ is uniformly bounded in $L^\infty(\mathbb{R}^n)$.

Let $x_k$ be the maximum point of $u_{s_k}$ and $u_{s_k}^{x_k}(x)=u_{s_k}(x+x_k)$. Since $u_{s_k}$ solves \eqref{eq:1.1} with $s=s_k$,  by \eqref{lem:3.2-2},
$$
u_{s_k}^{x_k}(0)\geq V_0^{\frac1{p-2}}>0,
$$
and for $\varphi\in C^\infty_0(\mathbb{R}^n)$, there holds
\begin{equation}\label{lem:3.2-3}
\int_{\mathbb{R}^n}(-\Delta)^{\frac{s_0}2}u_{s_k}^{x_k}(x)(-\Delta)^{s_k-\frac{s_0}2}\varphi\,dx+\int_{\mathbb{R}^n}V(x+x_k)u_{s_k}^{x_k}(x)\varphi\,dx
=\int_{\mathbb{R}^n}u_{s_k}^{x_k}(x)^{p-1}\varphi\,dx.
\end{equation}

We know from  Remark \ref{r:3.1}  that  $\{u_{s_k}^{x_k}(x)\}$ is bounded in $H^{s_0}(\mathbb{R}^n)$.

Using Lemma \ref{schauder} with $u=u_{s_k}^{x_k}(x)$ and $f=u_{s_k}^{x_k}(x)^{p-1}-V(x+x_k)u_{s_k}^{x_k}(x)$ , we see that  there exists $\alpha\in (0,1]$ such that
 \begin{equation*}
 \|u_{s_k}^{x_k}(x)\|_{C^{0,\alpha}(\mathbb{R}^n)} \leq C.
\end{equation*}
So we may assume that
\begin{equation*}
\begin{split}
&u_{s_k}^{x_k}(x)\rightharpoonup v_0 \text{ weakly in }H^{s_0}(\mathbb{R}^n);\\
&u_{s_k}^{x_k}(x)\to v_0 \text{ in } C_{loc}(\mathbb{R}^n).
\end{split}
\end{equation*}
As a result,  $v_0(0)\geq V_0^{\frac1{p-2}}>0$.

Now we distinguish two cases to derive a contradiction.

\bigskip

Case $1$. $\{x_k\}$ is bounded. We may assume that $x_k\to x_0$. Since
$$
(-\Delta)^{s_k-\frac{s_0}2}\varphi\to (-\Delta)^{\frac{s_0}2}\varphi\quad {\rm in}\quad L^2(\mathbb{R}^n)
$$
as $s_k \downarrow s_0$, taking limit in \eqref{lem:3.2-3},  we find that  $v_0$ solves
\begin{equation}\label{lem:3.2-4}
(-\Delta)^{s_0}v_0+V(x+x_0)v_0=v_0^{p-1} \text{  in  } \mathbb{R}^n.
\end{equation}
By \eqref{r:3.1-2}, \eqref{lem:3.2-4} and $x_k\to x_0$, we derive
\[
S(n,s_0)\leq \frac{\|(-\Delta)^{s_0/2}v_0\|_2^2}{\|v_0\|_p^2}<\|v_0\|_p^{p-2}\leq \lim_{k\to \infty}\|u_{s_k}^{x_k}\|_p^{p-2}=\lim_{s \downarrow s_0}\|u_s\|_p^{p-2}=S(n,s_0)
\]
a contadiction.

\bigskip

Case $2$. $\{x_k\}$ is unbounded. We may assume that $x_k\to \infty$. Then  $v_0$ solves
\begin{equation*}
(-\Delta)^{s_0}v_0+V_\infty v_0=v_0^{p-1} \text{  in  } \mathbb{R}^n.
\end{equation*}
Then we may obtain a contradiction as the case $1$. The proof is complete.
\end{proof}

\bigskip

Let  $v_s$ be defined in \eqref{scaling1}.

 Now, we  study the  asymptotic behaviour of $v_s$ as $s\downarrow s_0$.  A difficulty to be encountered is the lack the convergence of $\{v_s\}$ in the space of integrable functions as $s\downarrow s_0$. We then turn to use the pointwise convergence of $\{v_s\}$. Our result can be described as follows.

\bigskip

\begin{lemma} \label{lem2.2} There holds

\begin{equation}\label{lem2.2:a}
v_s\to Q_{s_0}\quad {\rm in}\quad L^p(\mathbb{R}^n)\quad {\rm and}\quad C_{loc}(\mathbb{R}^n)
\end{equation}
as $s\downarrow s_0$,  where $Q_{s}$ is given by \eqref{eq:1.9} and
 \begin{equation}\label{thm2.2-aa}
\lim_{s\downarrow s_0}\mu_{s}^{s-s_0}=1.
\end{equation}
\end{lemma}

\begin{proof}
 The fact that
\begin{equation*}
\|v_s\|_\infty=v_s(0)=1
\end{equation*}
and
\begin{equation}\label{thm1.2-1}
(-\Delta)^sv_s+\mu_s^{2s}V(\mu_s x+x_s)v_s=v_s^{p-1}\quad {\rm in} \quad \mathbb{R}^n
\end{equation}
allows us to show by Lemma \ref{schauder} that, there exists $\alpha\in(0,1)$ such that
\begin{equation*}
\|v_s\|_{C^{0,\alpha}(\mathbb{R}^n)}\leq C(\|v_s\|_\infty+\|v_s\|_\infty^{p-1})\leq C,
\end{equation*}
where $C>0$ depends only on $n$ and $s_0$. Hence, for any sequence $\{s_k\}$ with  $s_k\downarrow s_0$ as $k\to\infty$, correspondingly there exists a sequence solutions $\{v_{s_k}\}$ of \eqref{thm1.2-1} such that
\begin{equation*}
v_{s_k}\to v \quad {\rm  in } \quad C_{loc}(\mathbb{R}^n)
\end{equation*}
up to subsequence. Thus, $v(0)=1$, that is, $v$ is nontrivial.

Furthermore, for $\varphi\in C_c^\infty(\mathbb{R}^n)$, the Lebeasgue dominated theorem implies that as $k\to\infty$,
\begin{equation*}
\begin{split}
 \mu_{s_k}^{2{s_k}}\int_{\mathbb{R}^n}V(\mu_{s_k} x+x_{s_k})v_{s_k}\varphi\,dx\to0,\quad \int_{\mathbb{R}^n}v_{s_k}^{p-1}\varphi\,dx\to\int_{\mathbb{R}^n}v^{p-1}\varphi\,dx.
\end{split}
\end{equation*}

In order to prove the lemma, we need to pass the limit in \eqref{thm1.2-1}. This will  be done if we may show as $k\to\infty$,
\begin{equation*}
 \int_{\mathbb{R}^n}v_{s_k}(-\Delta)^{s_k}\varphi\,dx\to\int_{\mathbb{R}^n}v(-\Delta)^{s_0}\varphi\,dx.
\end{equation*}

To this end,  we may assume, without of generalization, that $\varphi\in C_c^\infty(\mathbb{R}^n)$ satisfying $\varphi=0$ if $|x|\geq R>1$ and $0<\varphi\leq 1$. By the definition of the fractional Laplacian  \eqref{eq:1.2}, we have for  $|x|\geq 2R>2$
and  $s\in(s_0,1)$ that,
\begin{equation}\label{thm1.2-a1}
\begin{split}
|(-\Delta)^s\varphi(x)|
=D(n,s)\int_{\{|x|\leq R\}}\frac{-\varphi(y)}{|x-y|^{n+2s}}\,dy
\leq C|x|^{-(n+2s_0)}.
\end{split}
\end{equation}

On the other hand, by Lemma 3.2 in \cite{NPV}, for any $x\in\mathbb{R}^n$,
\begin{equation*}
\begin{split}
(-\Delta)^s\varphi(x)=-\frac12D(n,s)\int_{\mathbb{R}^n}\frac{\varphi(x+y)+\varphi(x-y)-2\varphi(x)}{|y|^{n+2s}}\,dy.
\end{split}
\end{equation*}
We find for  $\delta\in (0,1-s_0)$ and $s\in (s_0,s_0+\delta)$ that a uniformly bound of $(-\Delta)^s\varphi(x)$, that is,
\begin{equation}\label{thm1.2-a3}
\begin{split}
|(-\Delta)^s\varphi(x)|
&\leq C\Bigg[\int_{\{|y|\leq1\}}\frac{\|D^2\varphi\|_{\infty}}{|y|^{n+2s_0+2\delta-2}}\,dy+\int_{\{|y|\geq1\}}\frac{1}{|y|^{n+2s_0}}\,dy\Bigg].
\end{split}
\end{equation}
By \eqref{thm1.2-a1}, \eqref{thm1.2-a3}, $0<\varphi\leq 1$ and  $0<v_s,v\leq 1$,
\begin{equation*}
|(-\Delta)^s\varphi(x)|
\leq \frac {C}{(1+ |x|)^{n+2s_0}}.
\end{equation*}
Hence,  we demonstrate by the Lebeasgue dominated theorem, as $k\to\infty$, that
\begin{equation*}
\int_{\mathbb{R}^n}v_{s_k}(-\Delta)^{s_k}\varphi\,dx\to\int_{\mathbb{R}^n}v(-\Delta)^{s_0}\varphi\,dx.
\end{equation*}
As a result,
\begin{equation*}
\int_{\mathbb{R}^n}v(-\Delta)^{s_0}\varphi\,dx=\int_{\mathbb{R}^n}v^{p-1}\varphi\,dx.
\end{equation*}
Noting that  $p=2_{s_0}^*$, we know from \cite{CLO} that $v=Q_{s_0}$, where $Q_{s_0}$ is given by \eqref{eq:1.9}.  Therefore,  Theorem 1.1 in \cite{CT} implies that
\begin{equation*}
S(n,s_0)=\frac{\|(-\Delta)^{s_0/2}v\|_2^2}{\|v\|_p^2}=\|v\|_p^{p-2}.
\end{equation*}

Consequently,
\begin{equation*}
\begin{split}
\liminf_{k\to\infty}\|v_{s_k}\|_p^p\geq \|v\|_p^p=S(n,s_0)^{\frac p{p-2}}.
\end{split}
\end{equation*}
By \eqref{r:3.1-2} and $\mu_s\leq1$, we have
\begin{equation*}
\begin{split}
\limsup_{k\to\infty}\|v_{s_k}\|_p^p
=\limsup_{k\to\infty}\mu_{s_k}^{\frac{2p}{p-2}(s_k-s_0)}\|u_{s_k}\|_p^p\leq \lim_{k\to\infty}\|u_{s_k}\|_p^p=S(n,s_0)^{\frac p{p-2}}.
\end{split}
\end{equation*}
Therefore,
\begin{equation*}
\lim_{k\to\infty}\|v_{s_k}\|_p^p=S(n,s_0)^{\frac p{p-2}}
\end{equation*}
and
\begin{equation*}
\lim_{k\to\infty}\mu_{s_k}^{s_k-s_0}=1.
\end{equation*}
This yields
\begin{equation*}
v_{s_k}\to v=Q_{s_0} \text{ strongly in } L^p(\mathbb{R}^n).
\end{equation*}
We conclude that \eqref{lem2.2:a} and \eqref{thm2.2-aa} hold since $\{s_k\}$ is arbitrary.
\end{proof}

\section{The concentration of ground state solutions}

In this section, we will determine the blowup rate and the location of the limit of the maximum point $x_s$  of  the ground state solution $u_s$ as $s\downarrow s_0$.

Let $u_s^{x_s}(x):=u_s(x+x_s)$, where $x_s$ is a maximum point of $u_s$.

\begin{lemma}\label{lem:4.3}
For fixed $R>0$, it holds that
\[
\lim_{s\downarrow s_0}\int_{\{|x|\geq R\}}(u_s^{x_s})^p(x)\,dx=0.
\]
\end{lemma}
\begin{proof}
Apparently,
\begin{equation}\label{lem:4.3-1}
\begin{split}
\int_{\{|x|\geq R\}}(u_s^{x_s})^p(x)\,dx=\mu_s^{-\frac{2p(s-s_0)}{p-2}}\int_{\{|y|\geq \mu_s^{-1}R\}}v_s^p(y)\,dy
\end{split}
\end{equation}
and
\begin{equation*}
\int_{\{|y|\geq \mu_s^{-1}R\}}v_s^p(y)\,dy\leq 2^p\Big(\int_{\{|y|\geq \mu_s^{-1}R\}}|v_s(y)-Q_{s_0}|^p\,dy+\int_{\{|y|\geq \mu_s^{-1}R\}}|Q_{s_0}|^p\,dy\Big).
\end{equation*}
In accordance to Lemmas \ref{lem:3.2} and  \ref{lem2.2},  $\mu_s\to0$ and $v_s\to Q_{s_0}$ in $L^p(\mathbb{R}^n)$ as $s\downarrow s_0$, we derive that
\[
\int_{\{|y|\geq \mu_s^{-1}R\}}v_s^p(y)\,dy \to 0
\]
as $s\downarrow s_0$.
Taking into account \eqref{thm2.2-aa}, we deduce from  \eqref{lem:4.3-1}  the required result.
\end{proof}

 \begin{proposition}\label{lem:4.1}
For any fixed $R>0$, it holds that
\[
\lim_{s\downarrow s_0}u_s^{x_s}(x)=0 \quad {\rm  uniformly \, \, for  }\quad x\in B_R^c(0).
\]

\end{proposition}

\begin{proof}
By \eqref{eq:1.1}, $u_s^{x_s}$ satisfies
\begin{equation*}
(-\Delta)^su_s^{x_s}=\Big((u_s^{x_s})^{p-2}-V(x+x_s)\Big)u_s^{x_s}\leq (u_s^{x_s})^{p-2} u_s^{x_s}\quad{\rm in }\quad\mathbb{R}^n.
\end{equation*}
Using Lemma  \ref{lem:2.4}, \ref{lem A.1} and \ref{lem:4.3}, we have for $r\in (0,R/4)$ and $x\in B_R^c(0)$ that
\begin{equation*}
\begin{split}
u_s^{x_s}(x)\leq \sup_{y\in B_{\frac{r}{2}}(x)}u_s^{x_s}(y)
\leq C\inf_{y\in B_{\frac{r}{2}}(x)}u_s^{x_s}(y)
\leq C\Big(\int_{B_{3r}^c(0)}(u_s^{x_s})^p(y)\,dy\Big)^{\frac{1}{p}}\to 0
\end{split}
\end{equation*}
as $s\downarrow s_0$.

\end{proof}

Now, we will study the asymptotic behaviour of $x_s$.  For this purpose, we need to establish uniformly decaying laws with respect to $s$  for $u_s^{x_s}$, $v_s$ and $Q_s$.
First, we have the following result.

\begin{lemma}\label{lem:4.4}
For $s>s_0$ and close to $s_0$, there exist $C_1, C_2, C_3>0$, independent of $s$, such that
\begin{equation*}
\|(-\Delta)^{s/2}v_s\|_2^2\leq C_1, \quad \|v_s\|_{2^*_s}^{2^*_s}\leq C_2, \quad \|v_s\|_{p}^{p}\leq C_3.
\end{equation*}
\end{lemma}

\begin{proof} The conclusion readily follows from Lemma \ref{lem2.2} and Remark \ref{r:3.1}.

\end{proof}

\bigskip

Next, we have the decaying law for $u_s^{x_s}$.
\begin{lemma}\label{lem:B.3}
For $s>s_0$, $s$ close to $s_0$, and fixed $R>0$, we have that
\[
u_s^{x_s}(x)=u_s(x+x_s)\leq C|x|^{-n}, \text{  } |x|\geq R,
\]
where $C>0$ is independent  of $s$.
\end{lemma}

\begin{proof}
 We rewrite \eqref{eq:1.1} as
\begin{equation}\label{lem:B.3-1}
(-\Delta)^su_s^{x_s}(x)+\frac{1}{2}V_0u_s^{x_s}(x)=((u_s^{x_s})^{p-2}(x)-V(x+x_s)+\frac{1}{2}V_0)u_s^{x_s}(x):=f_s(x).
\end{equation}
By Proposition  \ref{lem:4.1}, for $s>s_0$, $s$ close to $s_0$, if $|x|\geq R/2$ we have that
\begin{equation*}
(u_s^{x_s})^{p-2}(x)\leq \frac{1}{2}V_0\leq V(x+x_s)-\frac{1}{2}V_0,
\end{equation*}
that is,
$f_s(x)\leq 0$
 for  $x\in \mathbb{R}^n$ such that $|x|\geq \frac R2$.

Let $K_{s, \frac 12V_0}$ be the fundamental solution of
$$
(-\Delta)^su+\frac{1}{2}V_0u=0.
$$
By Lemma C.1 in \cite{FLS},
\[
0<K_{s, \frac 12V_0}(x)\leq C(n,s_0)V_0^{-1}|x|^{-n}
\]
provided that  $|x|>0$ and $s>s_0$.
Therefore, if $|x|\geq R$,
\begin{equation*}
\begin{split}
u_s^{x_s}(x)
%&=\int_{\mathbb{R}^n}K_{s, (1/2)P_0}(x-y)f_s(y)\,dy\\
&\leq \int_{\{|y|\leq \frac R2\}}K_{s, \frac 12V_0}(x-y)f_s(y)\,dy\leq \int_{\{|y|\leq \frac R2\}}\frac {C(n,s_0)}{V_0|x-y|^{n}}f_s(y)\,dy\leq \frac{C(n,s_0)}{2^{n}V_0|x|^{n}}\int_{\{|y|\leq \frac R2\}}f_s(y)\,dy.
\end{split}
\end{equation*}
By \eqref{r:3.1-2} and \eqref{lem:B.3-1}, we have
\begin{equation*}
\begin{split}
\int_{\{|y|\leq \frac R2\}}f_s(y)\,dy
\leq (1+2V_\infty)\int_{\{|y|\leq \frac R2\}}[1+(u_s^{x_s})^p(y)]\,dy\leq C.
\end{split}
\end{equation*}
It implies the conclusion.

\end{proof}

%\bigskip

Now, we turn to the study of the decaying law for $v_s$.
\begin{lemma}\label{lem:B.4}
For $s>s_0$ and $s$ close to $s_0$, there exists $C>0$ independent of $s$ such that
\begin{equation*}
Q_s(x)\leq \frac{C}{(1+|x|)^{n-2s}} \quad{\rm and } \quad v_s(x)\leq \frac{C}{(1+|x|)^{n-2s}}, \quad x\in \mathbb{R}^n.
\end{equation*}
\end{lemma}
\begin{proof}

The inequality for $Q_s$ comes from \eqref{eq:1.9}. So we only need to prove the inequality for $v_s$.

Since $v_s\leq1$, it suffices to prove that
\begin{equation*}
v_s(x)\leq C|x|^{-(n-2s)}, \quad x\in \mathbb{R}^n.
\end{equation*}
This inequality is equivalent to $\Bar{v}_s(x)\leq C$,
where $\Bar{v}_s$ is Kelvin transform of $v_s$, that is,
\begin{equation*}
\Bar{v}_s(x)=\frac{1}{|x|^{n-2s}}v_s(\frac{x}{|x|^2}).
\end{equation*}
By $0<v_s\leq 1$,
\begin{equation}\label{lem:B.4-3a}
\Bar{v}_s(x)\leq \frac{1}{|x|^{n-2s}}.
\end{equation}
It is well known that
\begin{equation*}
(-\Delta)^s\Bar{v}_s(x)=\frac{1}{|x|^{n+2s}}(-\Delta)^sv_s(\frac{x}{|x|^2}),
\end{equation*}
thus,
\begin{equation*}
(-\Delta)^s\Bar{v}_s(x)+|x|^{-4s}V(\mu_s\frac{x}{|x|^2}+x_s)\Bar{v}_s(x)=|x|^{-2p(s-s_0)}\Bar{v}_s^{p-1}(x)
\end{equation*}
and
\begin{equation*}
(-\Delta)^s\Bar{v}_s(x)\leq a(x)\Bar{v}_s,
\end{equation*}
where $a(x)=|x|^{-2p(s-s_0)}\Bar{v}_s^{p-2}(x)$. By \eqref{lem:B.4-3a}, we have
\begin{equation}\label{lem:B.4-5a}
0<a(x)\leq  |x|^{-2p(s-s_0)-(p-2)(n-2s)}.
\end{equation}

Let $r>0$. Then
\begin{equation*}
\int_{B_{2r}}a(x)^{\frac{n}{2s}}\,dx=\Big(\int_{\{|x|\leq \mu_s^2\}}+\int_{\{\mu_s^2\leq |x|\leq 2r\}}\Big)a(x)^{\frac{n}{2s}}\,dx:=I_1+I_2.
\end{equation*}

We first estimate $I_2$.

By the H\"{o}lder inequality, $s>s_0=\frac{(p-2)n}{2p}$ and \eqref{thm2.2-aa}, we have, for $s>s_0$  and close to $s_0$ that,
\begin{equation}\label{lem:B.4-7}
\begin{split}
I_2
&\leq \mu_s^{-\frac{p(s-s_0)n}{s}}\int_{\{\mu_s^2\leq |x|\leq 2r\}}\Bar{v}_s^{\frac{(p-2)n}{2s}}(x)\,dx
\leq C \left( \int_{B_{2r}} \Bar{v}_s^{p}(x) \, dx \right)^{\frac{(p-2)n}{2sp}}.
\end{split}
\end{equation}
By Lemma \ref{lem2.2:a}
and $0<v_s,{Q}_{s_0}\leq 1$, we demonstrate for $s>s_0$ and close to $s_0$ that,
\begin{equation*}
\begin{split}
\int_{\mathbb{R}^n}|\Bar{v}_s-\bar{Q}_s|^p\,dx
&=\int_{\{|x|\leq 1\}}|x|^{-2p(s-s_0)}|{v}_s-{Q}_s|^p\,dx+\int_{\{|x|\geq 1\}}|x|^{-2p(s-s_0)}|{v}_s-{Q}_s|^p\,dx\\
&\leq \Big[\int_{\{|x|\leq 1\}}|x|^{-4p(s-s_0)}\,dx\Big]^{1/2}2^p\Big[\int_{\{|x|\leq 1\}}|{v}_s-{Q}_s|^{p}\,dx\Big]^{1/2}+\int_{\{|x|\geq 1\}}|{v}_s-{Q}_s|^p\,dx\\
&\to 0
\end{split}
\end{equation*}
as $s\downarrow s_0$.
This inequality and the fact that
$$\lim_{s\downarrow s_0}\bar{Q}_{s}=\bar{Q}_{s_0}\quad {\rm in}\quad  L^p(\mathbb{R}^n)
$$
yield that
$$\lim_{s\downarrow s_0}\bar{v}_{s}=\bar{Q}_{s_0}\quad {\rm in}\quad  L^p(\mathbb{R}^n),
$$
where $\bar{Q}_{s_0}=|x|^{2s_0-n}Q_{s_0}(\frac{x}{|x|^2})$.  Hence,  we can choose $r$ small enough such that $I_2<\frac{1}2\delta_0$ via \eqref{lem:B.4-7}, here $\delta_0$ is again the same as that in Lemma \ref{lem A.1}.

Next, we estimate $I_1$. By Lemma \ref{lem:B.3} and $\lim_{s\downarrow s_0}\frac{(p-2)n}{2s}=p\geq2$, for $s>s_0$  and close to $s_0$, there exists $C>0$, independent of $s$, such that
\begin{equation}\label{lem:B.4-8}
\begin{split}
I_1
\leq C\mu_s^{n-\frac{(p-2)n^2}{2s}}\int_{\{|y|\geq \mu_s^{-2}\}}|y|^{\frac{p(s-s_0)n}{s}}|y|^{-\frac{(p-2)n^2}{2s}}\,dy=C\mu_s^{-n-\frac{2p(s-s_0)n}{s}+\frac{(p-2)n^2}{2s}}<\frac12\delta_0.
\end{split}
\end{equation}

It follows from \eqref{lem:B.4-7}, \eqref{lem:B.4-8},  Lemma \ref{lem A.1} and Lemma \ref{lem:4.4} that
\begin{equation}\label{lem:B.4-9-1}
\|\Bar{v}_s\|_{L^{\frac{(2_s^*)^2}{2}}(B_r)}\leq C_0\|\Bar{v}_s\|_{L^{2_s^*}(\mathbb{R}^n)}=C_0\|v_s\|_{L^{2_s^*}(\mathbb{R}^n)}\leq C,
\end{equation}
where  $C>0$ is independent of $s$.
For $s>s_0$  and close to $s_0$, there exists $q$, independent of $s$,  such that
\begin{equation}\label{lem:B.4-9}
\frac{(2_{s_0}^*)^2}{2(p-2)}>q>\frac{n}{2s_0}.
\end{equation}
For this $q$, we will estimate the integral
\begin{equation*}
\int_{B_{r}}a^{q}(x)\,dx=\Big(\int_{\{|x|\leq \mu_s^2\}}+\int_{\{\mu_s^2\leq |x|\leq r\}}\Big)a^{q}(x)\,dx:=J_1+J_2.
\end{equation*}

By Lemma \ref{lem:B.3}, for $s>s_0$  and close to $s_0$, there exists $C>0$, independent of $s$, such that
\begin{equation*}
\begin{split}
J_1
&=\mu_s^{2sq}\int_{\{|y|\geq \mu_s^{-2}\}}|y|^{2pq(s-s_0)}\bar{u}_s^{(p-2)q}(\mu_sy)\,dy
\leq C\mu_s^{2sq-2pq(s-s_0)+(p-2)qn-2n}\int_{\{|y|\geq 1\}}|y|^{2pq(s-s_0)-(p-2)qn}\,dy.
\end{split}
\end{equation*}
Observing that
$\lim_{s\downarrow s_0}\mu_s=0
$
 and
\[
\lim_{s\to (s_0)_+}2sq-2pq(s-s_0)+(p-2)qn-2n>n+(p-2)n\frac{n}{2s_0}-2n=(p-1)n>0,
\]
we get
\[
J_1\leq C.
\]

On the other hand,
by the H\"{o}lder inequality, \eqref{lem:B.4-9-1}, \eqref{lem:B.4-9} and \eqref{thm2.2-aa}, we have, for $s>s_0$  and close to $s_0$ that,
\begin{equation*}
\begin{split}
J_2\leq \mu_s^{-4pq(s-s_0)}\int_{\{\mu_s^2\leq |x|\leq 2r\}}\Bar{v}_s^{(p-2)q}(x)\,dx\leq C \left( \int_{B_{2r}} \Bar{v}_s^{\frac{(2_s^*)^2}{2}}(x) \, dx \right)^{\frac{2(p-2)q}{(2_s^*)^2}}\leq C.
\end{split}
\end{equation*}
Therefore, $$\int_{B_{r}}a^{q}(x)\,dx\leq C$$
uniformly with respect to $s$ for $s>s_0$  and close to $s_0$.

By \eqref{lem:B.4-5a}, \eqref{lem:B.4-9} and
\[
\lim_{s\downarrow s_0}2p(s-s_0)-(p-2)(n-2s)=4s_0,
\]
we have
$$\int_{\mathbb{R}^n\setminus B_{r}}a^{q}(x)\,dx\leq C.$$
Consequently,
$$\int_{\mathbb{R}^n}a^{q}(x)\,dx\leq C.$$

Now we can use Lemma \ref{lem:A.3} to conclude that
\[
\|\Bar{v}_s\|_{\infty}\leq  C \| \Bar{v}_s\|_{2_s^*}\leq C.
\]
The proof is completed.
\end{proof}

\bigskip

We are ready to locate the limit point  of $x_s$.

Let $\mathcal{V}_m=\{x\in \mathbb{R}^n: V(x)= \inf_{x\in \mathbb{R}^n}V(x)\}$ be the set of minimum points of $V$.
Although  $\mathcal{V}_m$ may contain many points, we show that $V(x_{s})$ converges to the same value. That is:
\begin{proposition}\label{lem:4.2} There holds that
\[
\lim_{s\downarrow s_0}V(x_{s})=\inf_{x\in \mathbb{R}^n}V(x).
\]
\end{proposition}

\begin{proof} Let $x_0\in \mathcal{V}_m$.
Since
\[
E_{V,s}(u_s)=\inf_{u\in H^s(\mathbb{R}^n)\setminus\{0\}}E_{V,s}(u)\leq E_{V,s}(u_s(\cdot+x_s-x_0)),
\]
and
\[
\|(-\Delta)^{s/2}u_s(\cdot+x_s-x_0)\|_2^2=\|(-\Delta)^{s/2}u_s\|_2^2,\quad \|u_s(\cdot+x_s-x_0)\|_p^p=\|u_s\|_p^p,
\]
 we deduce from  Lemma \ref{lem2.2} and the Lebesgue dominated convergence theorem that as $s\downarrow s_0$,
\begin{equation}\label{location-1}
\begin{split}
\int_{\mathbb{R}^n}V(x)u_s^2(x)\,dx&\leq \int_{\mathbb{R}^n}V(x)u_s^2(x+x_s-x_0)\,dx\\
&=\mu_s^{2s}\mu_s^{-\frac{2p}{p-2}(s-s_0)}\int_{\mathbb{R}^n}V(\mu_sx+x_0)v_s^2(x)\,dx\\
&=\mu_s^{2s}\Big(\inf_{x\in \mathbb{R}^n}V(x)\int_{\mathbb{R}^n}Q_{s_0}^2(x)\,dx+o(1)\Big)
\end{split}
\end{equation}

 For any sequence $\{s_k\}$ such that $s_k\downarrow s_0$ as $k\to\infty$, we claim that $\{x_{s_k}\}$ is bounded.

Suppose  on the contrary, there exists a subsequence,
denoted also by $\{s_k\}$, such that $\lim_{k\to\infty}|x_{s_k}|=+\infty$. By the fact that
\begin{equation}\label{location-2}
\begin{split}
\int_{\mathbb{R}^n}V(x)u_s^2(x)\,dx
&=\mu_s^{2s}\mu_s^{-\frac{2p}{p-2}(s-s_0)}\int_{\mathbb{R}^n}V(\mu_sx+x_s)v_s^2(x)\,dx,
\end{split}
\end{equation}
the assumption $(V1)$, Lemma \ref{lem2.2} and the Fatou lemma, we have as $s\downarrow s_0$,
\begin{equation*}
\begin{split}
\int_{\mathbb{R}^n}V(x)u_s^2(x)\,dx\geq\mu_s^{2s}\Big(V_\infty\int_{\mathbb{R}^n}Q_{s_0}^2(x)\,dx+o(1)\Big).
\end{split}
\end{equation*}
This inequality and  \eqref{location-1} yield
$$V_\infty\leq \inf_{x\in \mathbb{R}^n}V(x),$$
 which contradicts to the assumption $(V1)$.
Hence, $\{x_{s_k}\}$ is bounded. We may assume there exists $y_0\in \mathbb{R}^n$ such that $\lim_{k\to\infty}x_{s_k}=y_0$.

By \eqref{location-2}, the assumption $(V1)$, Lemma \ref{lem2.2} and the Fatou lemma, we have  as $s\downarrow s_0$,
\begin{equation*}
\begin{split}
\int_{\mathbb{R}^n}V(x)u_s^2(x)\,dx\geq\mu_s^{2s}\Big(V(y_0)\int_{\mathbb{R}^n}Q_{s_0}^2(x)\,dx+o(1)\Big)
\end{split}
\end{equation*}

Hence,  \eqref{location-1} and $V(y_0)\geq \inf_{x\in \mathbb{R}^n}V(x)$ implies
$$V(y_0)=\inf_{x\in \mathbb{R}^n}V(x).$$
As a result,
\[
\lim_{k\to\infty}V(x_{s_k})=\inf_{x\in \mathbb{R}^n}V(x).
\]
 Since  $\{s_k\}$ is arbitrary, the conclusion follows.
\end{proof}

\bigskip

By Theorem 5 in \cite{CW}, for the positive solution $u\in H^s(\mathbb{R}^n)\cap L^\infty(\mathbb{R}^n)$ of the equation
\[
(-\Delta)^su=f(x,u) \text{ in }\mathbb{R}^n,
\]
the Pohozaev identity
\begin{equation}\label{peq}
\frac{n-2s}{2}\int_{\mathbb{R}^n}f(x,u)u\,dx=\int_{\mathbb{R}^n}[nF(x,u)+x\cdot\nabla_xF(x,u)]\,dx
\end{equation}
holds provided that $f(x,u)u$, $F(x,u)$ and $x\cdot\nabla_xF(x,u)$ are in $L^\infty(\mathbb{R}^n)$.

Now we estimate the blowup rate of $\|u_s\|_{\infty}$ as $s\downarrow s_0$.
\begin{lemma}\label{blowup rate}
Let $n\geq 4$. Then there exists $A_{n,s_0}>0$ such that
\[
\lim_{s\downarrow s_0}(s-s_0)\|u_{s}\|^{p-2}_{\infty}=A_{n,s_0}\inf_{x\in \mathbb{R}^n}V(x).
\]
\end{lemma}

\begin{proof}

By the Pohozaev identity \eqref{peq},   we have
\begin{equation}\label{br-1}
\begin{split}
\Big(\frac{1}{p}-\frac1{2_s^*}\Big)\int_{\mathbb{R}^n}u_s^p(x)\,dx
=\int_{\mathbb{R}^n}[V(x)+\frac1{2s}x\cdot\nabla V(x)]u_s^2\,dx.
\end{split}
\end{equation}

 By Lemma \ref{lem2.2},  the left hand side of \eqref{br-1} can be written as
\begin{equation*}
\begin{split}
\Big(\frac{1}{p}-\frac1{2_s^*}\Big)\int_{\mathbb{R}^n}u_s^p(x)\,dx
&=\frac{s-s_0}{n}\mu_s^{\frac{2p}{p-2}(s_0-s)}\int_{\mathbb{R}^n}|v_s|^p(x)\,dx\\
&=\frac{s-s_0}{n}[ \int_{\mathbb{R}^n}|Q_{s_0}|^p(x)\,dx+o(1)]
\end{split}
\end{equation*}
as $s\downarrow s_0$.

On the other hand,
by Lemma \ref{lem2.2} and Proposition \ref{lem:4.2}, we may assume for any sequence $\{s_k\}$,  $s_k\downarrow s_0$, correspondingly $x_{s_k}\to x_0\in \mathbb{R}^n$ , we have $V(x_0)=\inf_{x\in \mathbb{R}^n}V(x)$   and $v_{s_k}\to Q_{s_0}$.

Since $n\geq 4$, Lemma \ref{lem2.2}, Lemma \ref{lem:B.4} and the Lebesgue dominated convergence theorem enable us to infer that
\begin{equation}\label{br-3}
\begin{split}
&\int_{\mathbb{R}^n}[V(x)+\frac1{2s_k}x\cdot\nabla V(x)]u_{s_k}^2\,dx\\
&=\mu_{s_k}^{2s_k}\mu_{s_k}^{-\frac{2p}{p-2}(s_k-s_0)}
\int_{\mathbb{R}^n}[V(x_{s_k}+\mu_{s_k}x)+\frac1{2s_k}(x_{s_k}+\mu_{s_k}x)\cdot\nabla V(x_{s_k}+\mu_{s_k}x)]v_{s_k}^2\,dx\\
&=\mu_{s_k}^{2s_k}\inf_{x\in \mathbb{R}^n}V(x)\Big[\int_{\mathbb{R}^n}Q_{s_0}^2\,dx+o(1)\Big]
\end{split}
\end{equation}
as $s\downarrow s_0$.
It follows from \eqref{br-1}--\eqref{br-3} that
\begin{equation*}
\begin{split}
\lim_{k\to\infty}(s_k-s_0)\mu_{s_k}^{-2s_k}
=n \inf_{x\in \mathbb{R}^n}V(x)\Big(\int_{\mathbb{R}^n}Q_{s_0}^p\,dx\Big)^{-1}\int_{\mathbb{R}^n}Q_{s_0}^2\,dx.
\end{split}
\end{equation*}
The assertion follows by the definition of $\mu_s$ and Lemma \ref{lem2.2}. The proof is complete.
\end{proof}

\bigskip

At the end of this section, we prove Theorem \ref{thm:1.1} and Corollary \ref{cor:1.1a}.
\bigskip

{\bf Proof of Theorem \ref{thm:1.1}}
By Lemma \ref{lem2.2}, we have $v_s\to Q_{s_0}$ in $L^p(\mathbb{R}^n)$  and $C_{loc}(\mathbb{R}^n)$ as $s\to s_0$. So Lemma \ref{lem:B.4} implies that $v_s\to Q_{s_0}$ uniformly in $\mathbb{R}^n$ as $s\to s_0$.
 The  proof of Theorem \ref{thm:1.1} then is completed by  Proposition \ref{lem:4.2}, Lemma \ref{lem:B.4} and  Lemma \ref{blowup rate}.\quad$\Box$

 \bigskip

{\bf Proof of Corollary \ref{cor:1.1a}}
By Theorem \ref{thm:1.1}, for any sequence $\{s_k\}$ such that  $s_k\downarrow s_0$, we have correspondingly   sequences $\{x_{s_k}\},\{v_{s_k}\}$ such that $x_{s_k}\to x_0\in \mathbb{R}^n$,   $v_{s_k}\to Q_{s_0}$ in  $L^p(\mathbb{R}^n)$, $\mu_{s_k}^{s_k-s_0}\to 1$  and $V(x_0)=\inf_{x\in \mathbb{R}^n}V(x)$.
Hence, for any $\varphi\in C_0^\infty(\mathbb{R}^n)$,  we get by \eqref{Qs0},
\begin{equation*}
\begin{split}
\lim_{k\to\infty}\int_{\mathbb{R}^n}|u_{s_k}|^p\varphi\,dx
&=\lim_{k\to\infty}\mu_{s_k}^{\frac{2p}{p-2}(s_0-{s_k})}\int_{\mathbb{R}^n}|v_{s_k}|^p(x)\varphi(\mu_{s_k}x+x_{s_k})\,dx\\
&=\varphi(x_0)\int_{\mathbb{R}^n}|Q_{s_0}|^p\,dx\\
&=S(n,s_0)^{\frac p{p-2}}\varphi(x_0).
\end{split}
\end{equation*}
The proof is completed.
\quad$\Box$

\bigskip

\bigskip

\section{Local uniqueness of ground state solutions}

\bigskip

In this section,  we devote to the proof of Theorem \ref{thm:1.3}. Under assumptions $(V1)-(V3)$, we first drive refined asymptotic behaviour of $x_s$.

\begin{lemma}\label{refined location}
Suppose $n\geq5$ and  $(V1)-(V3)$. Then
\begin{equation*}
\lim_{s\downarrow s_0}\frac{x_s-x_0}{\mu_s}=0.
\end{equation*}
\end{lemma}

\begin{proof}
By \eqref{location-1} and \eqref{location-2}, we have
\begin{equation}\label{refined location-1}
\int_{\mathbb{R}^n}V(\mu_sx+x_s)v_s^2(x)\,dx\leq \int_{\mathbb{R}^n}V(\mu_sx+x_0)v_s^2(x)\,dx.
\end{equation}

By the assumption $(V3)$, Lemmas \ref{lem2.2} and \ref{lem:B.4}, we deduce
\begin{equation}\label{refined location-2a}
\begin{split}
\int_{\mathbb{R}^n}[V(\mu_sx+x_0)-V(x_0)]v_s^2(x)\,dx
&= \int_{\mathbb{R}^n}\frac{V(\mu_sx+x_0)-V(x_0)}{|\mu_sx|^m}|\mu_sx|^mv_s^2(x)\,dx\\
&=A\mu_s^m \int_{\mathbb{R}^n}|x|^mQ_{s_0}^2(x)\,dx+o(\mu_s^m)
\end{split}
\end{equation}
as $s\downarrow s_0$ .
Equations \eqref{refined location-1} and \eqref{refined location-2a} then yield
\begin{equation}\label{refined location-2b}
\begin{split}
\int_{\mathbb{R}^n}[V(\mu_sx+x_s)-V(x_0)]v_s^2(x)\,dx
\leq A\mu_s^m \int_{\mathbb{R}^n}|x|^mQ_{s_0}^2(x)\,dx+o(\mu_s^m).
\end{split}
\end{equation}
On the other hand,
\begin{equation}\label{refined location-2}
\begin{split}
\int_{\mathbb{R}^n}[V(\mu_sx+x_s)-V(x_0)]v_s^2(x)\,dx
&=\mu_s^m\int_{\mathbb{R}^n}\frac{V(\mu_sx+x_s)-V(x_0)}{|\mu_sx+x_s-x_0|^m}{\big|x+\frac{x_s-x_0}{\mu_s}\big|^m}v_s^2(x)\,dx.
\end{split}
\end{equation}

Now, we claim that $\big|\frac{x_s-x_0}{\mu_s}\big|\leq C$.
On the contrary,  there would exist $\{s_k\}$, $s_k\downarrow s_0$, such that $\big|\frac{x_{s_k}-x_0}{\mu_{s_k}}\big|\to +\infty$. It follows from \eqref{refined location-2} and the Fatou lemma that, for any $C>0$ and $k$ large enough,
\begin{equation*}
\begin{split}
\int_{\mathbb{R}^n}[V(\mu_{s_k}x+x_{s_k})-V(x_0)]v_{s_k}^2(x)\,dx \geq C \mu_{s_k}^m,
\end{split}
\end{equation*}
which contradicts to \eqref{refined location-2b}. Hence,  $\big|\frac{x_{s}-x_0}{\mu_{s}}\big|\leq C$. So we may assume that for any $\{s_k\}$, $s_k\downarrow s_0$,   $\frac{x_{s_k}-x_0}{\mu_{s_k}}\to y_0\in \mathbb{R}^n$. Equation \eqref{refined location-2} then yields that
\begin{equation}\label{refined location-3}
\begin{split}
\int_{\mathbb{R}^n}[V(\mu_{s_k}x+x_s)-V(x_0)]v_{s_k}^2(x)\,dx = A\mu_{s_k}^m \int_{\mathbb{R}^n}|x+y_0|^mQ_{s_0}^2(x)\,dx+o(\mu_{s_k}^m).
\end{split}
\end{equation}
We see from  \eqref{refined location-2b} and \eqref{refined location-3} that $y_0=0$. The conclusion readily follows.
\end{proof}

\bigskip

From Theorem \ref{thm:1.1} and Lemma \ref{refined location}, we  immediately obtain the following  result.
\begin{proposition}\label{cor:7.1}Let $n\geq 5$,  $s_0<s<1$ and $2<p<2^*$.  Suppose  $(V1)-(V3)$. Denote
\begin{equation}\label{eta}
\eta_s=[A_{n,s_0}V(x_0)]^{-1/(2s)}(s-s_0)^{1/(2s)}.
\end{equation}
 Then
\begin{equation}\label{cor:7.1-1}
\lim_{s\downarrow s_0}\eta_s^{\alpha_s}u_s(\eta_s x+x_0)= Q_{s_0} \quad {\rm uniformly\, in } \quad x\in\mathbb{R}^n.
\end{equation}
Moreover, for $s>s_0$ and close to $s_0$,
\begin{equation*}
\eta_s^{{\alpha_s}}u_s(\eta_s x+x_0)\leq C(1+|x|^2)^{-(n-2s)},
\end{equation*}
where $C>0$ is independent of $s$.
\end{proposition}

Finally we  prove Theorem \ref{thm:1.3}.  We  commence with  the following Lemma.

\begin{lemma}\label{decay3}
Let $z_s\in H^s(\mathbb{R}^n)$ satisfy
\begin{equation*}
(-\Delta)^sz_s\leq C(1+|x|)^{-(n-2s)(p-2)}z_s.
\end{equation*}
Assume $0\leq z_s\leq 1$ and $\|z_s\|_{2_s^*}\leq C$, $C>0$ is independent of $s$. Then there exist $\delta_3>0$ and $C>0$, independent of $s$, such that if $s_0\leq s<s_0+\delta_3$, then
\begin{equation*}
z_s(x)\leq C(1+|x|)^{-(n-2s)},\quad |x|>0.
\end{equation*}
\end{lemma}
\begin{proof}

By the Kelvin transformation,
\[
Z_s(x)=\frac1{|x|^{n-2s}}z_s\big(\frac{x}{|x|^2}\big)
\]
satisfies
\begin{equation*}
(-\Delta)^sZ_s(x)\leq a(x)Z_s(x),
\end{equation*}
where
\begin{equation*}
a(x)=C|x|^{(p-2)(n-2s)-4s}(1+|x|)^{-(n-2s)(p-2)}\leq C\min\{|x|^{[(p-2)(n-2s)-4s]}, |x|^{-4s}\}.
\end{equation*}

Fix $q>\frac n{2s_0}$.
It is easy to see that  $a$ is  uniformly bounded in $L^q(\mathbb{R}^n)$.

By Lemma \ref{lem:A.3},
\begin{equation*}
\begin{split}
\|Z_s\|_\infty\leq C\|Z_s\|_{2_s^*}=C\|z_s\|_{2_s^*}\leq C.
\end{split}
\end{equation*}
As a result,
\begin{equation*}
z_s(x)\leq C|x|^{-(n-2s)},\quad |x|>0.
\end{equation*}
The proof is complete since  $0<z_s\leq 1$.
\end{proof}

\bigskip

{\bf Proof of Theorem \ref{thm:1.3}.} \ \ Suppose $u_{s,1}$ and $u_{s,2}$ are two ground state solutions of \eqref{eq:1.1}. We will show that there exists $\delta_3>0$ such that if $s_0<s<s_0+\delta_3$, then $u_{s,1}\equiv u_{s,2}$.

Suppose on the contrary that there exists a sequence $\{s_k\}$,  $s_k\downarrow s_0$ as $k\to\infty$, such that $u_{s_k,1}\neq u_{s_k,2}$ for $k$ large enough. Set
\begin{equation*}
w_k(x)=v_{k,1}-v_{k,2},
\end{equation*}
where $v_{k,i}=\eta_{s_k}^{\alpha_{s_k}}u_{s_k,i}(\eta_{s_k}x+x_0)$, $i=1,2$, satisfying
\begin{equation}\label{thm:1.3-1a}
(-\Delta)^{s_k}v_{k,i}+\eta_{s_k}^{2s_k}V(\eta_{s_k}x+x_0)v_{k,i}=v_{k,i}^{p-1}
\end{equation}
and $\eta_{s_k}$ is defined in \eqref{eta}.

 Since $w_k$ are not equal to zero for $k$ large enough, we define  $W_k=\frac{w_k}{\|w_k\|_\infty}$. Then $\|W_k\|_{L^\infty(\mathbb{R}^n)}=1$ and  $W_k$ satisfies the following equation
\begin{equation}\label{thm:1.3-2}
\begin{split}
(-\Delta)^{s_k}W_k+\eta_{s_k}^{2{s_k}}V(\eta_{s_k}x+x_0)W_k
=\frac{v_{k,1}^{p-1}-v_{k,2}^{p-1}}{v_{k,1}-v_{k,2}}{w}_k.
\end{split}
\end{equation}
By Proposition \ref{cor:7.1}, we have
\begin{equation}\label{decay3-56}
\begin{split}
\Big|\frac{v_{k,1}^{p-1}-v_{k,2}^{p-1}}{v_{k,1}-v_{k,2}}\Big|
&=\Bigg|\frac{(p-1)\int_0^1[tv_{k,1}+(1-t)v_{k,2}]^{p-2}\,dt}{v_{k,1}-v_{k,2}}\Bigg|\\
&\leq (p-1)2^{p-2}(v_{k,1}^{p-2}+v_{k,2}^{p-2})\\
&\leq C(1+|x|)^{-(p-2)(n-2s)}.
\end{split}
\end{equation}

The Kato inequality(see \cite{FLS})
\[
(-\Delta)^s|u|\leq sgn(u)(-\Delta)^su,
\]
and \eqref{thm:1.3-2}, \eqref{decay3-56} yield
\begin{equation*}
\begin{split}
(-\Delta)^{s_k}|W_k|
\leq sgn(W_k)(-\Delta)^{s_k}W_k
\leq \frac{v_{k,1}^{p-1}-v_{k,2}^{p-1}}{v_{k,1}-v_{k,2}}|W_k|
\leq C(1+|x|)^{-(p-2)(n-2s)}|W_k|.
\end{split}
\end{equation*}
Lemma \ref{decay3} then implies that
\begin{equation}\label{thm:1.3-3}
|W_k|\leq C(1+|x|^2)^{-\frac{n-2s}{2}}.
\end{equation}
Whence by Lemma \ref{schauder} and $|W_k|\leq 1$, there exists $W_0\in C^{0,\beta}(\mathbb{R}^n)$ and a subsequence of $\{W_k\}$ such that $W_k\to W_0$ in $C_{loc}^{0,\beta}(\mathbb{R}^n)$,
and  $W_0$ satisfies
\begin{equation*}
(-\Delta)^{s_0}W_0=(2_{s_0}^*-1)Q_{s_0}^{2_{s_0}^*-2}W_0.
\end{equation*}
It follows from Theorem 1.1 in \cite{DPS} that, there exist $\lambda_1, \lambda_2,\cdots, \lambda_{n+1}$ such that
\begin{equation}\label{thm:1.3-5}
W_0=\sum_{j=1}^n\lambda_j\partial_{x_j}Q_{s_0}+\lambda_{n+1}\big(\frac{n-2s_0}{2}Q_{s_0}+x\cdot \nabla Q_{s_0}\big).
\end{equation}
So  \eqref{thm:1.3-3} and \eqref{thm:1.3-5} yield  $W_k\to W_0$ uniformly in  $\mathbb{R}^n$ and $\|W_0\|_\infty=1$.

\bigskip

Next, we will prove that $\lambda_j=0$, $j=1,2,\cdots, n+1$.  Once this is true, we  obtain  $W_0=0$, which contradicts to $\|W_0\|_\infty=1$, and  the conclusion immediately follows.

By \eqref{thm:1.3-1a} and the Pohozaev identity \eqref{peq}, we have
\begin{equation*}
\frac{1}{2}\int_{\mathbb{R}^n}V(\eta_{s_k}x+x_0)v_{k,i}^2\,dx
+\frac{1}{2n}\eta_{s_k}\int_{\mathbb{R}^n}x\cdot \nabla V(\eta_{s_k}x+x_0)v_{k,i}^2\,dx
=\frac{s_k-s_0}{n}\eta_{s_k}^{-2s_k}\int_{\mathbb{R}^n}v_{k,i}^p\,dx.
\end{equation*}
Hence,
\begin{equation}\label{thm:1.3-8}
\begin{split}
&\frac{1}{2}\int_{\mathbb{R}^n}V(\eta_{s_k}x+x_0)(v_{k,1}+v_{k,1})(v_{k,1}-v_{k,1})\,dx\\
&=-\frac{1}{2n}\eta_{s_k}\int_{\mathbb{R}^n}x\cdot \nabla V(\eta_{s_k}x+x_0)(v_{k,1}+v_{k,1})(v_{k,1}-v_{k,1})\,dx\\
&+\frac{s_k-s_0}{n}\eta_{s_k}^{-2s_k}\int_{\mathbb{R}^n}\frac{v_{k,1}^p-v_{k,2}^p}
{v_{k,1}-v_{k,2}}(v_{k,1}-v_{k,2})\,dx.
\end{split}
\end{equation}
Letting $k\to\infty$ in \eqref{thm:1.3-8}, we find by \eqref{eta} --\eqref{cor:7.1-1} that
\begin{equation*}
\begin{split}
V(x_0)\int_{\mathbb{R}^n}Q_{s_0}W_0\,dx=\frac{pA_{n,s_0}}{n}\int_{\mathbb{R}^n}Q_{s_0}^{p-1}W_0\,dx.
\end{split}
\end{equation*}
We deduce by \eqref{thm:1.3-5}, $p=2_{s_0}^*$ and
$$
\int_{\mathbb{R}^n}Q_{s_0}\partial_{x_j}Q_{s_0}\,dx=\int_{\mathbb{R}^n}Q_{s_0}^{p-1}\partial_{x_j}Q_{s_0}\,dx=0
$$ that
\begin{equation}\label{thm:1.3-10}
\begin{split}
\lambda_{n+1}\Big[V(x_0)\int_{\mathbb{R}^n}Q_{s_0}(\frac{n-2s_0}{2}Q_{s_0}+x\cdot \nabla Q_{s_0})\,dx-\frac{pA_{n,s_0}}{n}\int_{\mathbb{R}^n}Q_{s_0}^{p-1}(\frac{n-2s_0}{2}Q_{s_0}+x\cdot \nabla Q_{s_0})\,dx\Big]=0.
\end{split}
\end{equation}
A integration by part implies
\begin{equation*}
\begin{split}
\int_{\mathbb{R}^n}Q_{s_0}(\frac{n-2s_0}{2}Q_{s_0}+x\cdot \nabla Q_{s_0})\,dx
=\int_{\mathbb{R}^n}\frac{n-2s_0}{2}Q_{s_0}^2+\frac12x\cdot \nabla Q_{s_0}^2\,dx=-s_0\int_{\mathbb{R}^n}Q_{s_0}^2\,dx<0.
\end{split}
\end{equation*}
Similarly,
\begin{equation}\label{thm:1.3-12}
\begin{split}
\int_{\mathbb{R}^n}Q_{s_0}^{p-1}(\frac{n-2s_0}{2}Q_{s_0}+x\cdot \nabla Q_{s_0})\,dx
&=\int_{\mathbb{R}^n}\frac{n-2s_0}{2}Q_{s_0}^p+\frac1px\cdot \nabla Q_{s_0}^p\,dx=0.
\end{split}
\end{equation}
It follows from \eqref{thm:1.3-10}--\eqref{thm:1.3-12} that $\lambda_{n+1}=0$.

Multiplying \eqref{thm:1.3-1a} by $v_{k,i}$ and integrating by parts, we have
\begin{equation*}
\int_{\mathbb{R}^n}\partial_{x_j}(V(\eta_{s_k}x+x_0))v_{k,i}^2\,dx=0,\quad i=1,2;\quad j=1,2,\cdots,n.
\end{equation*}

Therefore,
\begin{equation*}
\int_{\mathbb{R}^n}\partial_{x_j}(V(\eta_{s_k}x+x_0))(v_{k,1}+v_{k,2})(v_{k,1}-v_{k,2})\,dx=0,\quad i=1,2;\quad j=1,2,\cdots,n
\end{equation*}
and
\begin{equation}\label{thm:1.3-14}
\int_{\mathbb{R}^n}\frac{\partial_{x_j}(V(\eta_{s_k}x+x_0))}{|\eta_{s_k}x|^{m-2}\eta_{s_k}x_j}|x|^{m-2}x_j(v_{k,1}+v_{k,2})W_k\,dx=0,\quad i=1,2;\quad j=1,2,\cdots,n.
\end{equation}
Since $\lambda_{n+1}=0$,  $Q_{s_0}(x)=Q_{s_0}(|x|)$ and $Q_{s_0}'(|x|)<0$ for $x\neq0$, letting $k\to \infty$, in \eqref{thm:1.3-14}, we have by \eqref{cor:7.1-1} and $(V3)$ that
\begin{equation*}
\begin{split}
0&=\int_{\mathbb{R}^n}|x|^{m-2}x_jQ_{s_0}W_0\,dx=\int_{\mathbb{R}^n}|x|^{m-2}x_jQ_{s_0}\sum_{i=1}^n\lambda_i\partial_{x_i}Q_{s_0}\,dx
=\lambda_j\int_{\mathbb{R}^n}|x|^{m-3}x_j^2Q_{s_0}Q'_{s_0}(|x|)\,dx.
\end{split}
\end{equation*}
Consequently, $\lambda_1=\lambda_2=\cdots=\lambda_n=0$.

$\Box$

\bigskip

\section{Appendix}

\bigskip

\begin{lemma}\label{lem:A1}
Let  $0<s\leq s_0$. Suppose  $\inf_{x\in \mathbb{R}^n}V>0$ and $V\in L_{loc}^\infty(\mathbb{R}^n)$. Then $S_V(n,s)$ defined in \eqref{eq:1.4} is not achieved.
\end{lemma}
\begin{proof}
Let $\psi\in C_c(\mathbb{R}^n)$, $\psi\gneq 0$. For  $\lambda>0$, $\psi_\lambda(x)=\psi(\lambda x)$ verifies
\begin{equation*}
\begin{split}
E_{V,s}(\psi_\lambda)=\frac{\int_{\mathbb{R}^n}|(-\Delta)^{s/2}\psi|^2\,dx+\lambda^{-2s}\int_{\mathbb{R}^n}V\psi^2\,dx}
{\lambda^{n-2s-\frac{2n}{p}}\Big(\int_{\mathbb{R}^n}|\psi|^p\,dx\Big)^{2/p}}.
\end{split}
\end{equation*}

If $0<s< s_0$, equivalently,  $p>2_s^*$, we have  $n-2s-\frac{2n}{p}>0$. Therefore,
$$
\lim_{\lambda\to +\infty}E_{V,s}(\psi_\lambda)=0,
$$
and we deduce further that   $S_V(n,s)=0$. This means that  $S_V(n,s)$, $0<s< s_0$, is not achieved. Indeed, for any  minimizer $\varphi\in H^s(\mathbb{R}^n)\setminus\{0\}$ of $S_V(n,s)$, we have  $E_{V,s}(\varphi)=0$ and thus $\varphi=0$, a contradiction.

If $s= s_0$, or $p=2_s^*$, take $\psi_s=(1+|x|^2)^{\frac{2s-n}{2}}$,   by \eqref{eq:1.7} and \eqref{eq:1.9a}, we see that $\psi_{\lambda(x),s}=\psi_s(\lambda x)$,  $\lambda>0$, satisfies
  \begin{equation*}
\begin{split}
E_{V,s}(\psi_{\varepsilon,s})=\frac{\int_{\mathbb{R}^n}|(-\Delta)^{s/2}\psi_s|^2\,dx+\lambda^{-2s}\int_{\mathbb{R}^n}V\psi_s^2\,dx}
{\Big(\int_{\mathbb{R}^n}|\psi_s|^p\,dx\Big)^{2/p}}=S(n,s)+\frac{\lambda^{-2s}\int_{\mathbb{R}^n}V\psi_s^2\,dx}
{\Big(\int_{\mathbb{R}^n}|\psi_s|^p\,dx\Big)^{2/p}},
\end{split}
\end{equation*}
which yields
$$
\lim_{\lambda\to +\infty}E_{V,s}(\psi_{\lambda_s})=S(n,s).
$$
 The fact
 $$
 S(n,s)\leq S_V(n,s)\leq E_{V,s}(\psi_{\lambda_s})
 $$
 gives $S_V(n,s)=S(n,s)$. If there exists a minimizer $\varphi$ of $S_V(n,s_0)$ for $s= s_0$, we obtain
\[
S(n,s)=E_{V,s}(\varphi)> E_{0,s}(\varphi)\geq S(n,s)
\]
a contradiction.
\end{proof}

\bigskip

{\bf Data Availability} {There is no data in the paper.}

{\bf Statements and Declarations}  { There is no conflict of interest for all authors.}

{\bf Acknowledgements} {  Jinge Yang was supported by NNSF of China, No:12361025 and  by Jiangxi
Provincial Natural Science Foundation, No:20232BAB201002 , 20252BAC230002. Jianfu Yang is supported by NNSF of China, No:12171212.
}

\end{document}